\documentclass{article}
\usepackage{graphicx} % Required for inserting images
\usepackage{amsmath, bm}
\usepackage{amsfonts}
\usepackage{soul}
\usepackage{tikz, tikz-3dplot, tikzsymbols}
\usepackage{cases}
\usepackage{amsthm}
\usepackage{comment}
\usepackage{hyperref}
\usepackage{tabularx}
\usepackage{pgfplots}
\usepackage{mathrsfs}
\pgfplotsset{compat=1.18}
\usepackage{float}
\usepackage{authblk}
\usetikzlibrary{calc,arrows.meta,fillbetween}
\pgfdeclarelayer{pre main}
\pgfdeclarelayer{main}
\pgfsetlayers{pre main,main}
\usepackage{subcaption}
\usepackage[para]{footmisc}
\usepackage{enumitem}
\usepgfplotslibrary{fillbetween}
\usepackage[margin=1in]{geometry}
\usepackage[backend=biber,maxnames=99, minnames=99, sorting=none]{biblatex}
\newcommand{\pdfrac}[2]{\frac{\partial #1}{\partial #2}}

\newcommand{\R}{\mathbb{R}}

\newcommand{\Hk}[1]{H^{#1}(\Omega)}
\newcommand{\Hkrad}[1]{H_{\mathrm{rad}}^{#1}(\Omega)}
\newcommand{\Lrad}{L_{\mathrm{rad}}^2(\Omega)}
\newcommand{\PV}{\textrm{P.V.}}
\newcommand{\res}[2]{\mathop{\mathrm{Res}}_{#1 = #2}}

\theoremstyle{remark}
\newtheorem{claim}{Claim}
\newtheorem{remark}{Remark}

\theoremstyle{plain} % Set the style for theorems
\theoremstyle{definition} % Set the style for definitions
\newtheorem{definition}{Definition}

\DeclareMathOperator*{\argmin}{argmin}

\usetikzlibrary{calc,decorations.markings}
\tikzset{
  ->-/.style={postaction={decorate},
              decoration={markings, mark=at position 0.5 with {\arrow{>}}}},
}
\tikzset{
  ->--/.style={dashed, postaction={decorate},
              decoration={markings, mark=at position 0.5 with {\arrow{>}}}},
}
\tikzset{
  -<-/.style={postaction={decorate},
              decoration={markings, mark=at position 0.5 with {\arrow{<}}}}
}

\counterwithin{equation}{section}

\title{Dynamic Response of a Finite Circular Plate on an Elastic Half-Space Using the Truncated Lamb Kernel}
\date{\today}
\author[1]{Greyson Meares \thanks{gcm00010@mix.wvu.edu}}
\author[2]{Sage Meiling \thanks{shm00001@mix.wvu.edu}}
\author[2]{Charis Tsikkou \thanks{tsikkou@math.wvu.edu}}
\affil[1]{Lane Department of Computer Science and Electrical Engineering, West Virginia University, Morgantown, WV 26506, USA}
\affil[2]{School of Mathematical and Data Sciences, West Virginia University, Morgantown, WV 26506, USA}

\begin{document}
\maketitle
\begin{abstract}\noindent
We develop an exact operator formulation for the dynamic interaction between a finite circular elastic plate and an elastic half–space.  Classical analyses of this system, beginning with Lamb’s representation of the half–space response, assume an infinite plate and rely on the diagonalization of the soil operator by the continuous Hankel transform. However, for a plate of finite radius $R$, the applied traction and resulting displacement are supported only on $0\le r\le R$, thus the governing operator is the spatially truncated Lamb operator
\[
  \mathcal{M}(\omega)=\chi_{[0,R]}\,T(\omega)\,\chi_{[0,R]},
\]
where $T(\omega)$ is the standard Hankel multiplier involving the Rayleigh 
denominator $\Omega(\xi,\omega)$. Truncation destroys the diagonal structure of $T(\omega)$ and introduces real–axis singularities due to the Rayleigh pole, in addition to the square–root branch points at $\xi=k_T$ and $\xi=k_L$. We express the action of $\mathcal{M}(\omega)$ on a finite-disk Bessel basis $\{\phi_n(r)=A_{1,n}J_0(\lambda_n r)+A_{2,n}I_0(\lambda_n r)\}$, which satisfies the free-edge boundary conditions of the plate and obtain formulas for the resulting matrix elements, which involve integrals of the Lamb kernel. These are evaluated as Cauchy principal values, with contributions from residues representing radiation damping in the half–space. The resulting matrix is dense, meaning that its modal matrix on the basis of Bessel finite-disks has nonzero entries for almost all pairs of indices $(m,n)$, but is spectrally convergent. Inversion of this matrix provides a complete solution for a finite radius in the frequency domain. Our analysis reproduces Chen et al.'s \cite{Chen1988} classical finite-radius experiment for small $R$, and our solution approaches the infinite–radius solution for large $R$, while also quantifying the finite–radius corrections. To the best of our knowledge, this work provides the first exact operator-level treatment of the finite–radius plate–half–space interaction that retains the full Lamb kernel and its nonlocal structure.
\end{abstract}

\vspace{5mm}

\noindent
{\bf Keywords.} Plate–Half-Space Interaction; Elastic Half-Space; Truncated Lamb Operator; Hankel Transform; Rayleigh Waves; Modal Coupling; Spectral Quadrature; Singular Integrals; Finite-Disk Bessel Basis.
\\

\vspace{5mm}

\noindent
{\bf AMS Subject Classifications.} Primary: 74J05, 74B05, 45E10; Secondary: 74K20, 35P05, 45P05, 65T40.

\section{Introduction}
The dynamic interaction between elastic structures and deformable foundations is a well-known problem with important applications in impact mechanics, 
soil–structure interaction, seismology, and nondestructive testing. Lamb~\cite{Lamb1904} established the response of an elastic half-space to 
normal surface loading and demonstrated that the associated frequency-domain Green's function includes the characteristic square-root factors
\[
  \sqrt{\xi^{2}-k_{T}^{2}}
  \qquad\text{and}\qquad
  \sqrt{\xi^{2}-k_{L}^{2}},
\]
where $k_{T}=\omega/C_{T}$ and $k_{L}=\omega/C_{L}$ are the shear- and compressional-wave numbers. These terms become purely imaginary for 
$\xi < k_{T}$ or $\xi < k_{L}$, corresponding to evanescent shear and 
compressional fields, and become real for $\xi > k_{T}$ or $\xi > k_{L}$, indicating the beginning of propagating waves in the underlying elastic 
half-space. Therefore, the points
\[
  \xi = k_{T}=\omega/C_{T},
  \qquad 
  \xi = k_{L}=\omega/C_{L},
\]
are genuine square-root branch points of the spectral response of the 
half-space. In Lamb’s formulation, the axial displacement under an applied axisymmetric surface traction is expressed in the Hankel and Fourier domains as
\begin{equation}\label{eq1}
  \widehat{w}^{\,0}(\xi,\omega)
  = \alpha_{\mathrm{HS}}(\xi,\omega)\,\widehat{q}^{\,0}(\xi,\omega),
\end{equation}
where
\begin{equation}
\label{eq2}
  \alpha_{\mathrm{HS}}(\xi,\omega)
  = -\frac{\alpha(\xi,\omega)\,k_{T}^{2}}{\mu\,\Omega(\xi,\omega)},
  \qquad
  \alpha(\xi,\omega)=\sqrt{\xi^{2}-k_{L}^{2}},
\end{equation} and 
\begin{equation}\label{eq3}
  \Omega(\xi,\omega)
  = \left( 2\xi^{2} - k_{T}^{2} \right)^{2}
      - 4 \xi^{2} 
        \sqrt{\xi^{2}-k_{T}^{2}}\,
        \sqrt{\xi^{2}-k_{L}^{2}}
\end{equation}
is the Rayleigh denominator. Our notation follows standard conventions in the plate-half-space literature (see, e.g.,~\cite{Chen1988}): $w(r,t)$ denotes the normal surface
displacement, $q(r,t)$ the interfacial normal traction, and $\omega$ the angular frequency. We note that  $\alpha=\sqrt{\xi^{2}-k_{L}^{2}}$, in Chen et al., denotes the longitudinal depth-decay wavenumber, which appears as a prefactor in the surface admittance $\alpha_{\mathrm{HS}}$, while $\Omega(\xi,\omega)$ carries the Rayleigh pole and branch-point structure.

Throughout the paper, we write $\widehat{f}$ for the (Cartesian) Fourier 
transform of $f$, and we use the superscript ${}^0$ to denote the Hankel transform of zeroth-order in radial coordinate. Thus,
\[
  \widehat{f}(\omega)
  = \int_{-\infty}^{\infty} f(t)\,e^{i\omega t}\,dt,
  \qquad
  (\mathscr{H}_0f)(\xi) = f^{0}(\xi)
  = \int_{0}^{\infty} f(r)\,J_{0}(\xi r)\,r\,dr.
\]
The cusps observed when plotting $\Omega(\xi,\omega)$ along the real axis 
are the direct manifestation of the branch points at $\xi=k_{T}$ and $\xi=k_{L}$. Therefore, all nonlocal, dispersive and radiative features of the half-space response are encoded in the multiplier 
$\alpha_{\mathrm{HS}}(\xi,\omega)$.

We emphasize that in the axisymmetric setting, the radial wavenumber $\xi$ is inherently nonnegative. Unlike Cartesian Fourier variables, the zeroth-order Hankel transform is defined for $ \xi \in [0,\infty),$ so negative wavenumbers do not arise. Although the analytic continuation of $\Omega(\xi,\omega)$ has symmetric branch points at $\xi = \pm k_{T}$ and $\xi = \pm k_{L}$, only positive branch points are relevant for the physical problem and for the integral representations that appear in the truncated operator to be introduced below. Consequently, the principal value integrals and Rayleigh-type poles encountered in the finite-radius formulation arise solely from the positive real axis.

To formulate the finite-radius problem, we recall that in the frequency domain the response of the elastic half-space response is governed by an operator $T(\omega)$, which maps an axisymmetric normal traction $q(r,\omega)$ to the corresponding displacement $w(r,\omega)$ through the Lamb kernel. Although $T(\omega)$ admits diagonalization by the Hankel transform in $[0,\infty)$, this property is lost once the interaction is restricted to a finite disk. Since only $0\le r\le R$ is in contact with the half-space, we introduce the characteristic function  
\[
  \chi_{[0,R]}(r)
  =
  \begin{cases}
    1, & 0\le r\le R,\\[4pt]
    0, & r>R,
  \end{cases}
\]
the radial characteristic function of the disk, and define the physically relevant truncated operator
\[
  \mathcal{M}(\omega) 
  = \chi_{[0,R]}\, T(\omega)\, \chi_{[0,R]},
\]
which maps tractions (normal surface force densities) supported on $[0,R]$ to displacements supported on the same interval. This operator is the central object of the finite-radius formulation developed in the present work.

A particularly influential application of the untruncated operator $T(\omega)$ appears in the work of Chen et al.~\cite{Chen1988}, who analyzed the low-velocity impact of a thin elastic plate resting on an elastic half-space. By assuming that the plate extends across the entire surface (the infinite-radius approximation), 
Chen et al. were able to diagonalize $T(\omega)$ using the continuous Hankel transform, leading to a closed-form frequency-domain solution. This assumption plays a crucial role: placing both the displacement and the traction on $[0,\infty)$, so that the soil operator becomes a simple Hankel multiplier.

In practice, many physical systems involve structural components in contact with the ground that have a \emph{finite radius}. Typical examples include circular foundations, impact plates, sensors, load cells, and test fixtures. For finite plate geometries, the infinite-domain assumption is broken down. Since both the plate deflection and applied traction are supported only in $0 \le r \le R$, the truncated operator $\mathcal{M}(\omega)=\chi_{[0,R]}T(\omega)\chi_{[0,R]}$ is not diagonalizable by the Hankel transform, and the analytic simplifications used in the infinite-radius theory by Chen et al.~\cite{Chen1988} are no longer available. Despite the extensive literature on plates on elastic foundations, a rigorous formulation of the \emph{finite-radius} plate–half-space problem using the exact Lamb kernel does not appear to exist.

A natural point of comparison is the work of Schmidt and 
Krenk~\cite{SchmidtKrenk1982}, who studied the asymmetric vibrations of a circular plate of finite radius in an elastic half–space using the exact Lamb half–space solution.  Their mechanical model differs from the present one in that the plate is described by Mindlin’s first–order shear deformation theory and thus incorporates transverse–shear deformation and rotary inertia, whereas the present study adopts the classical Kirchhoff-Love model. On the mathematical side, Schmidt and Krenk formulate a direct integral equation on $0\le r\le R$ in which plate displacement and contact traction are unknown.  The half–space reaction is represented through an influence kernel derived from Bycroft’s version of Lamb’s solution, and the resulting first–kind integral equation is solved by expanding the contact stress in Popov’s polynomial basis and reducing the problem to a finite 
algebraic system for the expansion coefficients.

We study the finite-radius setting from an operator perspective by working with the truncated Lamb operator
$\mathcal{M}(\omega)=\chi_{[0,R]}T(\omega)\chi_{[0,R]},$
which acts on radial functions supported on the disk $0\le r\le R$. We regard $\mathcal{M}(\omega)$ as an operator on the Hilbert space $L^2(0,R)$ endowed with the standard radial inner product $\langle f,g\rangle=\int_0^R f(r)\,g(r)\,r\,dr.$
While Schmidt and Krenk already capture the finite–radius geometry and solve directly for contact traction, their analysis does not exploit the spectral properties of the truncated Lamb operator that play a central role in the present approach. Here, we develop the first exact operator–level formulation of the finite–radius problem in the framework of Lamb’s kernel. We expand all fields on the disk in the complete Bessel eigenbasis
$\{\phi_{n}(r)= A_{1,n}J_0(\lambda_n r)+A_{2,n}I_0(\lambda_n r)\}$,
which satisfies the boundary conditions 
\begin{equation}
\label{eqbc}
\begin{aligned}
\frac{\partial}{\partial r}\left(\frac{\partial^2 \widehat{w}}{\partial r^2}+\frac{1}{r} \cdot \frac{\partial \widehat{w}}{\partial r}\right)&\biggr|_{r=R}=0,\\
\left(\frac{\partial^2 \widehat{w}}{\partial r^2}+\nu_p \frac{1}{r}\cdot \frac{\partial \widehat{w}}{\partial r}\right)&\biggr|_{r=R}=0,
\end{aligned}
\end{equation} where $\nu_p$ denotes the Poisson ratio of the circular elastic plate. The eigenvalues $\lambda_{n}$ and coefficients $A_{1,n},A_{2,n}$, will be derived in Section~\ref{subsec:eigenfunctions}. On this basis, the truncated operator $\mathcal{M}(\omega)$ admits a dense matrix representation whose entries involve integrals of the Lamb kernel containing the Rayleigh–wave singularities. These integrals are interpreted in the sense of Cauchy principal values, with explicit residue terms that account for radiation damping into the elastic half-space.

Our analysis highlights a fundamental distinction between the finite- and infinite-radius settings. In the infinite case, $T(\omega)$ acts as a diagonal multiplier in Hankel space, whereas for finite radius, the truncated operator $\mathcal{M}(\omega)$ is fully nonlocal and couples all Bessel modes. Our analysis reproduces Chen et al.'s finite-radius experiment for small $R$, and we observe convergence to the infinite-radius model as $R$ becomes large, while quantifying the finite-radius corrections. The contributions of this paper are:
\begin{enumerate}
  \item \emph{Exact finite–radius operator formulation:} construction of the 
        truncated Lamb operator 
        $$\mathcal{M}(\omega)=\chi_{[0,R]}T(\omega)\chi_{[0,R]}$$ 
        and its modal representation.
  \item \emph{Finite–disk Bessel expansion:} a spectrally accurate basis for 
        representing all plate fields on $[0,R]$ satisfying boundary 
        conditions.
  \item \emph{Principal–value integrals and Rayleigh poles:} rigorous evaluation 
        of the $\xi$ -integrals that define $\mathcal{M}(\omega)$, including the Cauchy principal 
        values and the contributions of residues associated with radiation damping.
  \item \emph{Numerical inversion of the truncated operator:} stable computation of a well-conditioned modal matrix representation of $\mathcal{M}(\omega)$ which is then inverted numerically.
  \item \emph{Comparison with the infinite–plate model:} convergence to Chen et al.’s solution as $R\to\infty$ and quantitative characterization of all finite–radius corrections.
\end{enumerate}

The paper is organized as follows. In Section~2, we present the governing equations for the plate and the elastic half-space. Section~3 introduces the finite-disk Bessel basis 
and constructs the truncated Lamb operator. Section~4 analyzes the principal value structure and the Rayleigh poles. Section~5 develops the modal system and the inversion procedure. Section~6 contains numerical examples in comparison with infinite-plate results. Finally, in Section~7 we summarize our main results and suggest avenues for future investigation.

\section{Governing Equations}
We consider a thin, axisymmetric circular plate of radius $R$ and thickness $h$ that is in continuous contact with the surface $z=0$ of a linearly elastic half–space occupying $\{z<0\}$. The displacement of the plate is denoted by $w(r,t)$ and by the continuous contact assumption is equal to the displacement of the soil in the contact region. The normal traction exerted by the plate in the half–space is denoted by $q(r,t)$ and is continuous for $r \in (0, R)$ and zero for $r>R$. Following Chen et al.~\cite{Chen1988}, we adopt the Kirchhoff-Love model with small strain for the plate and the classical Lamb formulation for the half–space. 

\subsection{Plate Equation}\label{subsec:plate_equation}

Under a centered impact force $p(t)$ and the unknown half-space reaction $q(r,t)$, the axisymmetric transverse deflection satisfies
\begin{equation}\label{eq:plate_time}
D\,\Delta^2 w(r,t) + \rho h\,\pdfrac{^2w}{t^2}
= \frac{p(t)}{2\pi r}\delta(r) - q(r,t), \qquad 0\le r\le R,\; t>0,
\end{equation}
where $\Delta=\partial_{rr}+\frac{1}{r}\partial_r$ is the radial Laplacian and
$D=\frac{E h^3}{12(1-\nu_p^2)}$ is the flexural rigidity.
The plate is free at $r=R$, so $w$ satisfies the standard free-edge conditions \eqref{eqbc}, together with regularity at $r=0$.

Taking the Fourier transform in time yields the frequency-domain plate equation
\begin{equation}\label{eq:plate_freq}
\Big(D\,\Delta^2 - \rho h\,\omega^2\Big)\widehat w(r,\omega) + \widehat q(r,\omega)
= \frac{\widehat p(\omega)}{2\pi r}\delta(r), \qquad 0\le r\le R.
\end{equation}
The remaining relation between $\widehat q$ and $\widehat w$ is provided by Lamb's half-space solution.

\subsection{Half–Space Equation}
\label{sec:HS}
The response of the elastic half-space to a normal surface traction is described by Lamb’s classical solution. Evaluated at the surface $z=0$, the Hankel-transformed surface displacement and traction satisfy
\begin{equation}
\label{eq:LambSurface}
  \widehat{w}^{\,0}(\xi,\omega)
  = \alpha_{\mathrm{HS}}(\xi,\omega)\,\widehat{q}^{\,0}(\xi,\omega),
\end{equation}
where $\alpha_{\mathrm{HS}}(\xi,\omega)$ is the half-space admittance given by~\eqref{eq2}; see Lamb~\cite{Lamb1904} and, for example, Chen et al.~\cite{Chen1988}. $\alpha_{HS}$ is the fundamental solution of the half-space in the normal direction.

The square–root branch points occur at $\xi=k_{L}=\omega/C_{L}$ and $\xi=k_{T}=\omega/C_{T}$, separating the evanescent and propagating regimes of the compressional and shear fields.

\medskip\noindent
\textbf{Remark on sign conventions.}
We follow the expression of Chen et al. in~\cite{Chen1988}, which contains an overall minus sign on the right-hand side of (11). This sign reflects Chen et al.’s convention that positive traction acts \emph{downward} in the half–space while positive displacement is taken \emph{upward}.  

\subsection{Infinite-Domain Soil and the Finite-Radius Surface Operators}
Combining the plate equation with the Lamb relation 
\eqref{eq:LambSurface} shows that the interaction between the plate and the half–space is fully captured by the soil compliance operator
\[
  T(\omega):\widehat{q}(r,\omega)\mapsto \widehat{w}(r,\omega),
\]
which maps a prescribed normal traction $\widehat{q}(\cdot,\omega)$ on the surface 
to the corresponding normal displacement $\widehat{w}(\cdot,\omega)$.  
In the Hankel–frequency domain, this operator is diagonal:
\begin{equation}
\label{eq:T-Hankel}
  \widehat{w}^{\,0}(\xi,\omega)
  = \alpha_{\mathrm{HS}}(\xi,\omega)\,\widehat{q}^{\,0}(\xi,\omega).
\end{equation}
Equivalently, in physical space, we may write
\begin{equation}
\label{eq:T-realspace}
  \widehat{w}(r,\omega)
  = (T(\omega)\widehat{q})(r)
  = \int_{0}^{\infty} K_{\omega}(r,s)\,\widehat{q}(s,\omega)\,s\,ds,
\end{equation}
where the integral kernel $K_{\omega}$ is defined by
\begin{equation}
\label{eq:Kernel-def}
  K_{\omega}(r,s)
  = \int_{0}^{\infty} 
      \alpha_{\mathrm{HS}}(\xi,\omega)\,
      J_{0}(\xi r)J_{0}(\xi s)\,\xi\,d\xi.
\end{equation}
Thus $T(\omega)$ is a nonlocal integral operator on the radial half–line $0\le r<\infty$, whose diagonalization by the Hankel transform is given by 
\eqref{eq:T-Hankel}.
For a plate of finite radius $R$, the applied traction is supported on $0\le r\le R$, and the physically relevant operator is the truncation of $T(\omega)$ to the disk:
\begin{equation}
\label{eq:A-def}
  \mathcal{M}(\omega)
  = \chi_{[0,R]}\,T(\omega)\,\chi_{[0,R]}.
\end{equation}
Equivalently,
\begin{equation}
\label{eq:A-realspace}
  (\mathcal{M}(\omega)\widehat{q})(r)
  =\chi_{[0,R]}(r) \int_{0}^{R} K_{\omega}(r,s)\,\widehat{q}(s,\omega)\,s\,ds,
\end{equation}
so $\mathcal{M}(\omega)$ is an integral operator in $[0,R]$ with a kernel obtained by restricting $T(\omega)$ to the finite domain.

From an operator-theoretic perspective, the Lamb half-space operator $T(\omega)$ may be viewed as a frequency-dependent Neumann-to-Dirichlet map for the elastic half-space.  In the frequency domain, for a prescribed axisymmetric normal surface traction $\widehat q(r,\omega)=\widehat{\sigma}_{zz}(r,0,\omega)$ on the boundary $z=0$, the operator $T(\omega)$ returns the corresponding normal displacement $\widehat w(r,\omega)$. Here, $\sigma_{zz}$ denotes the normal component of the Cauchy stress tensor.

In the infinite domain setting, this operator acts on functions defined on the half–line $[0,\infty)$ and admits a diagonal representation under the Hankel transform via the Lamb surface relation. In the finite–radius configuration studied here, both plate displacement and applied traction are confined to the disk $0\le r\le R$, and interaction with the half–space is restricted to this bounded region. The physically relevant operator is therefore the truncated Neumann–to–Dirichlet map $\mathcal{M}(\omega)=\chi_{[0,R]}T(\omega)\chi_{[0,R]},$
which maps the tractions supported on the disk to the displacements supported on it.

Since truncation breaks Hankel diagonalization, the operator
$\mathcal{M}(\omega)$ is nonlocal on $[0,R]$: the displacement at a given radius depends on the traction over the entire disk.  As a consequence, the Galerkin matrix associated with $\mathcal{M}(\omega)$ exhibits global mode coupling and has no sparsity structure; in particular, its matrix representation on a modal basis is, in general, dense.

\section{Finite-Disk Bessel Eigenfunctions}
Due to the finite radius of the plate, the unknown deflection $w$ is naturally posed on the disk
$r\in[0,R]$ and must satisfy the free-edge conditions at $r=R$ together with regularity at the origin. Many numerical formulations enforce boundary conditions only indirectly, for example, through Lagrange multipliers, penalty terms, or weak
imposition in a finite–element setting. By contrast, our spectral expansion uses eigenfunctions that exactly satisfy the boundary conditions. As a result, every truncated approximation automatically inherits the correct behavior at both $r=0$ and $r=R$, without auxiliary constraints.

%\textcolor{red}{Rather than imposing these constraints through other means (like what\textbf{??}), we adopt a spectral expansion whose basis functions satisfy the conditions exactly.}

Specifically, we use the eigenmodes of the axisymmetric radial
biharmonic operator along with the free-edge conditions. These modes are the standard ``circular plate modes'': they form
a natural basis for Galerkin discretization and diagonalize the plate stiffness term. In addition, in the radial geometry the eigenmodes admit Bessel-type closed forms, which
will be useful later when expressing the half-space contribution.

\subsection{Eigenfunctions on \texorpdfstring{$[0,R]$}{[0,R]}}
\label{subsec:eigenfunctions}
We represent the deflection on the finite interval as the series $\widehat{w}(r, \omega) = \sum_{n=1}^\infty a_n(\omega) \phi_n(r)$,
where $\left\{\phi_n\right\}$ are the solutions to the eigenvalue problem
along with boundary conditions~\eqref{eqbc}
\[
\Delta^2 \phi_n = \lambda_n^4 \phi_n \qquad \text{for } \qquad 0<r<R,
\]
With $\Delta = \partial_{rr} + \tfrac{1}{r}\partial_r$ being the radial, axisymmetric Laplacian. These modes have two key properties: (i) they satisfy the free-edge
conditions at $r=R$ together with regularity at the origin ($\phi_n'(0) = 0$ and $\phi_n(0)$ finite) by construction,
and (ii) they admit explicit Bessel-type representations. A derivation can be found in ~\cite{asmar}:
\[
\phi_n(r) = A_{1n} J_0(\lambda_n r) + A_{2n} I_0(\lambda_n r),
\]
The latter property plays a crucial analytical role in the formulation.  Since the half-space operator is most naturally expressed in terms of a truncated Hankel transform, the coupling involves integrals of the form
\[
\int_0^R \phi_n(r) J_0(\xi r)\, r\,dr
\]
and their asymptotic behavior as $\xi\to\infty$. The Bessel structure of the biharmonic eigenmodes allows these Hankel-type integrals to be evaluated in closed form using well-known Bessel function identities. The constants $\lambda_n, A_{1n}, A_{2n}$ are determined solely by the boundary conditions. %~\eqref{eqbc}.
Plugging the eigenfunctions into the free-edge conditions yields the system which determines $A_{1n}, A_{2n}$:
\[
\begin{bmatrix}
    c_{1n} & c_{2n} \\
    c_{3n} & c_{4n}
\end{bmatrix}
\begin{bmatrix}
    A_{1n} \\
    A_{2n}
\end{bmatrix} = 0,
\]
where
\begin{align*}
    c_{1n} &= \lambda_n^3 J_0'''(\lambda_n R) + \frac{\lambda_n^2}{R} J_0''(\lambda_n R) - \frac{\lambda_n}{R^2}J_0'(\lambda_n R) \\
    c_{2n} &= \lambda_n^3 I_0'''(\lambda_n R) + \frac{\lambda_n^2}{R} I_0''(\lambda_n R) - \frac{\lambda_n}{R^2}I_0'(\lambda_n R) \\
    c_{3n} &= \lambda_n^2 J_0''(\lambda_n R) + \frac{\nu_p \lambda_n}{R} J_0'(\lambda_n R) \\
    c_{4n} &= \lambda_n^2 I_0''(\lambda_n R) + \frac{\nu_p \lambda_n}{R} I_0'(\lambda_n R)
\end{align*}
Here $J_0'(\lambda_n R)$ and $I_0'(\lambda_n R)$ denote derivatives of the Bessel functions and modified Bessel functions with respect to their argument, evaluated at $\lambda_n R$. 
For nontrivial solutions, the determinant of the coefficient matrix must be equal to zero, resulting in the characteristic equation for $\lambda_n$:
\[
    c_{1n} c_{4n} - c_{2n} c_{3n} = 0.
\]
%The eigenvalues $\lambda_n$ are the roots of this equation, which are computed numerically.
The solutions to this characteristic equation yield the countable set of real eigenvalues.

\subsection{Orthogonality, Completeness, and Normalization}

We will repeatedly project functions onto the basis $\{\phi_n\}$ and evaluate integrals involving products $\phi_m\phi_n$. For this reason, we note the completeness, orthogonality, and normalization properties used
throughout. Our goal is to show that the free-edge eigenmodes defined in the previous section form a complete basis in $L^2_{\mathrm{rad}}(\Omega)$. To this end, we introduce the Kirchhoff-Love plate energy form, construct the
associated self-adjoint operator, show that its resolvent is compact, and then apply the spectral theorem for compact self-adjoint operators. Let 
\[
    \Omega = \{\rm{x} \in \mathbb{R}^2: |\rm{x}| < R \}.
\]
For $u,v\in H^2(\Omega)$, define the symmetric bilinear form,
\begin{align}
    \label{bform_cart}
    &a : H^2(\Omega) \times H^2(\Omega) \to \mathbb{R} \\
    &a(u,v)
    :=\int_\Omega\Big(
    \left(1 - \nu_p\right)(u_{xx}v_{xx} + 2u_{xy}v_{xy} + u_{yy}v_{yy}) 
    + \nu_p (\Delta u)(\Delta v)
    \Big)\,d\textbf{x}
\end{align}
where derivatives are understood in the weak sense. The associated natural boundary conditions are free-edge conditions at $r=R.$ We equip $H^2(\Omega)$ with the standard norm:
\begin{equation}
    \| u \|_{H^2(\Omega)}^2 = \| u \|_{L^2(\Omega)}^2 + \| \nabla u \|_{L^2(\Omega)}^2 + \| \mathcal{D}^2 u \|_{L^2(\Omega)}^2
\end{equation}
where $\mathcal{D}^2u$ denotes the Hessian of $u$ (Note: Physically, $\nu_p$ is Poisson's ratio and typically takes a value between $0$ and $0.5$).
We will first need the following inequality to show the boundedness of the bilinear form.
\begin{claim}
    \[ \| \Delta u \|_{L^2(\Omega)} \leq c \| u \|_{H^2(\Omega)} \]
    for some constant $c>0$ and all $u \in H^2(\Omega)$. 
\end{claim}
\begin{proof}
    \begin{align*}
        \| \Delta u \|_{L^2(\Omega)}
        &= \| \partial_{xx} u + \partial_{yy} u \|_{L^2(\Omega)} \\
        &\leq \| \partial_{xx} u \|_{L^2(\Omega)} + \| \partial_{yy} u \|_{L^2(\Omega)}
    \end{align*}
    From here, we use the elementary inequality $a + b \leq \sqrt{2} \sqrt{a^2 + b^2}$ for $a,b\geq 0$ to obtain
    \begin{align*}
        &\| \partial_{xx} u \|_{L^2(\Omega)} + \| \partial_{yy} u \|_{L^2(\Omega)} \\
        &\leq \sqrt{2} \sqrt{ \| \partial_{xx} u \|_{L^2(\Omega)}^2 + \| \partial_{yy} u \|_{L^2(\Omega)}^2 } \\
        &\leq \sqrt{2} \sqrt{ \| \mathcal{D}^2 u \|_{L^2(\Omega)}^2 } \\
        &\leq c \| u \|_{H^2(\Omega)}
    \end{align*}
\end{proof}

\begin{claim}
    \label{boundedness}
    The bilinear form $a(\cdot,\cdot)$ is bounded on $H^2(\Omega)$.
\end{claim}
\begin{proof}
    \begin{align}
    |a(u,v)| &\leq \int_{\Omega}\Big(
        \nu_p |\Delta u| |\Delta v|
        + (1-\nu_p) \big(|u_{xx}||v_{xx}| + 2|u_{xy}||v_{xy}| + |u_{yy}||v_{yy}|\big)
    \Big)\,d\textbf{x} \\
    &\leq \int_{\Omega} |\Delta u| |\Delta v| \,d\textbf{x}+2\| \mathcal{D}^2 u \|_{L^2(\Omega)}\| \mathcal{D}^2 v \|_{L^2(\Omega)} \\
    &= \| \Delta u \|_{L^2(\Omega)} \| \Delta v \|_{L^2(\Omega)}+2\| \mathcal{D}^2 u \|_{L^2(\Omega)}\| \mathcal{D}^2 v \|_{L^2(\Omega)} \\
    &\leq (c^2+2)  \| u \|_{H^2(\Omega)} \| v \|_{H^2(\Omega)}
    \end{align}
\end{proof}
where the Cauchy–Schwarz inequality was used in (3.4) and (3.5). Because bounded bilinear forms on normed spaces are automatically continuous, Claim 2 also implies that $a(\cdot,\cdot)$ is continuous on $H^2(\Omega).$

To obtain an operator representation via the form method, we first show that
$a(\cdot,\cdot)$ is a \emph{closed} symmetric bilinear form.  Let
\[
D(a):=H^2_{\mathrm{rad}}(\Omega)
\]
denote the form domain, equipped with the inner product
\[
\langle u,v\rangle_a := a(u,v) + \langle u,v\rangle_{L^2(\Omega)},
\]
with associated form norm
\[
\|u\|_a^2 := a(u,u) + \|u\|_{L^2(\Omega)}^2.
\]
We denote by $H_a$ the space $H^2_{\mathrm{rad}}(\Omega)$ endowed with this inner product.
The form $a$ is said to be \emph{closed} if $H_a$ is complete with respect to
$\|\cdot\|_a$, i.e., if $(H_a,\langle\cdot,\cdot\rangle_a)$ is a Hilbert space.

As a set, $H_a$ coincides with $H^2_{\mathrm{rad}}(\Omega)$, but it is equipped with the form norm $\|\cdot\|_a$ rather than the standard $H^2$ norm.
Here we define the radial subspaces:
\[
H^2_{\mathrm{rad}}(\Omega)
:= H^2(\Omega)\cap L^2_{\mathrm{rad}}(\Omega),
\qquad
L^2_{\mathrm{rad}}(\Omega)
:= \{u\in L^2(\Omega): u(R\mathbf{x})=u(\mathbf{x})
\text{ for all } R\in SO(2)\}.
\]
The space $L^2_{\mathrm{rad}}(\Omega)$ is a closed subspace of $L^2(\Omega)$,
and consequently $H^2_{\mathrm{rad}}(\Omega)$ is a closed subspace of
$H^2(\Omega)$ consisting of radial functions.

We next introduce the notion of convergence with respect to the form norm and recall the definition of a closed form.
\begin{definition}[$a$-Cauchy and $a$-convergence]
Let $a$ be a symmetric bilinear form with domain $D(a)\subset H$, and define
the form norm
\[
\|u\|_a^2 := a(u,u) + \|u\|_H^2.
\]
A sequence $\{u_n\}\subset D(a)$ is said to be \emph{$a$-Cauchy} if
\[
\|u_n-u_m\|_a \to 0
\quad \text{as } n,m\to\infty.
\]
We say that $u_n$ \emph{$a$-converges} to $u\in H$, and write
$u_n\xrightarrow{a}u$, if $u_n\to u$ in $H$ and $\{u_n\}$ is $a$-Cauchy.
\end{definition}
\begin{definition}[Closed form]
The bilinear form $a(\cdot,\cdot)$ is said to be \emph{closed} if whenever
$u_n\xrightarrow{a}u$ for some $u\in H$, then $u\in D(a)$ and
\[
\|u_n-u\|_a \to 0
\quad \text{as } n\to\infty.
\]
\end{definition}
Equivalently, the form $a(\cdot,\cdot)$ is closed if and only if
$(D(a),\|\cdot\|_a)$ is complete; see~\cite[Thm.~VI.1.11]{kato}.
\begin{claim}
    $a(\cdot,\cdot)$ is a closed form. i.e., $H_a$ is complete.
\end{claim}
\begin{proof}
    To show that $H_a$ is complete, we first demonstrate that the form norm $\|\cdot\|_a$ is equivalent to the standard $H^2$ norm $\|\cdot\|_{H^2(\Omega)}$.
    From Claim~\ref{boundedness}, and the positivity $a(u,u)\ge 0$ (for $0<\nu_p<1$) there exists a constant $C_B>0$ such that
    \begin{equation}
        0\leq a(u,u) \leq C_B \| u \|_{H^2(\Omega)}^2, \quad \forall u \in H^2(\Omega)
    \end{equation}
    Thus, we have the inequality
    \begin{equation}
        \label{a_bound}
        \| u \|_a^2 = a(u,u) + \| u \|_{L^2(\Omega)}^2 \leq (C_B+1) \| u \|_{H^2(\Omega)}^2, \quad \forall u \in H^2(\Omega)
    \end{equation}
    Next, we need to bound the $H^2$ norm by the form norm.
    We know that for any $u \in H^2(\Omega)$, $||u||_{H^2(\Omega)}^2 = ||u||_{L^2(\Omega)}^2 + ||\nabla u||_{L^2(\Omega)}^2 + ||\mathcal{D}^2 u||_{L^2(\Omega)}^2$ 
    and we note the inequality from %Demkowicz's Lemma 3.4.2 
    \cite[Lemm~3.4.2]{Demkowicz}
    \begin{equation}
        \| u \|^2_{H^2(\Omega)} \leq C_a \left( \left|\int_{\Omega} u \, d\textbf{x} \right|^2 + \left|\int_{\Omega} \partial_x u \, d\textbf{x} \right|^2 + \left|\int_{\Omega} \partial_y u \, d\textbf{x} \right|^2 + \| \mathcal{D}^2 u \|_{L^2(\Omega)}^2 \right)
    \end{equation}
    but when $u = U(r) \in \Hkrad{2}$, the integrals of the first derivatives vanish due to the radial symmetry
    \[
        \int_{\Omega} \partial_x u \, d\textbf{x} = \int_{\Omega} U'(r) \cos \theta \, d\textbf{x} = \int_0^{2\pi} \cos \theta \, d\theta \int_0^R U'(r) r \, dr = 0
    \]
    and similarly for $\partial_y u$.
    Thus, for $u \in \Hkrad{2}$, we have
    \[
        \| u \|^2_{H^2(\Omega)} \leq C_a \left( \left|\int_{\Omega} u \, d\textbf{x} \right|^2 + \| \mathcal{D}^2 u \|_{L^2(\Omega)}^2 \right)
        \leq C_a \left(|\Omega| \|u\|^2_{L^2(\Omega)} + \| \mathcal{D}^2 u \|_{L^2(\Omega)}^2 \right)
    \]
    Where $|\Omega| = \pi R^2$. Therefore, we conclude that
    \begin{equation}
        \exists C'_a>0 \text{ such that } \| u \|_{H^2(\Omega)}^2 \leq C'_a \left( \| u \|_{L^2(\Omega)}^2 + \| \mathcal{D}^2 u \|_{L^2(\Omega)}^2 \right), \quad \forall u \in \Hkrad{2}.
    \end{equation}
    On the other hand, using the definition of $\|\cdot\|_a$, we have
    \begin{align}
        &\|u\|^2_{L^2(\Omega)} + a(u,u) = \|u\|^2_{L^2(\Omega)} + \int_\Omega \left( (1-\nu_p)(u_{xx}^2 + 2u_{xy}^2 + u_{yy}^2) + \nu_p (\Delta u)^2 \right) d\textbf{x} \\
        &\geq \|u\|^2_{L^2(\Omega)} + (1-\nu_p) \| \mathcal{D}^2 u \|_{L^2(\Omega)}^2 \\ 
        &\geq  (1-\nu_p) \left( \|u\|^2_{L^2(\Omega)} + \| \mathcal{D}^2 u \|_{L^2(\Omega)}^2 \right),  \quad \forall u \in H^2(\Omega) \text{ and } \nu_p \in (0, 1)\\
        &\geq \frac{1-\nu_p}{C'_a} \| u \|_{H^2(\Omega)}^2, \quad \forall u \in \Hkrad{2}
    \end{align}
    hence, combining (3.9) and (3.15), we have the equivalence of norms
    \begin{equation}
    \label{eq:norm_eq}
        \frac{1}{C_B+1} \| u \|_a^2 \leq \| u \|_{H^2(\Omega)}^2 \leq \frac{C'_a}{1-\nu_p} \| u \|_a^2, \quad \forall u \in \Hkrad{2}
    \end{equation}
    Let $\{u_n\}\subset H_a $ be Cauchy in the norm $ \|\cdot\|_a$.
    By the norm-equivalent estimate~\ref{eq:norm_eq} we obtain for all $n, m$: $\| u_n-u_m \|_{H^2(\Omega)}^2 \leq \frac{C'_a}{1-\nu_p} \| u_n-u_m \|_a^2$. Thus $\{u_n\}$ is Cauchy in $\Hkrad{2}$.
    Since $\Hkrad{2}$ is complete, there exists $u\in \Hkrad{2}$ such that $\|u_n-u\|_{H^2(\Omega)}\to 0$.
    Finally, by the lower bound $\|u_n-u\|_a^2 \le C'\,\|u_n-u\|_{H^2(\Omega)}^2$, we conclude $\|u_n-u\|_a\to 0$.
    This show that $H_a$ is complete, and therefore the bilinear form $a(\cdot,\cdot)$ is closed.
\end{proof}
Since $a(\cdot,\cdot)$ is symmetric, nonnegative, and closed on $L^2_{\mathrm{rad}}(\Omega)$ with domain $H^2_{\mathrm{rad}}(\Omega)$, Kato's representation theorem
\cite[Thms.~VI.2.1 \& VI.2.6]{kato} yields a unique self-adjoint operator
$A$ on $L^2_{\mathrm{rad}}(\Omega)$ such that
\[
a(u,v)=\langle Au,v\rangle_{L^2(\Omega)}
\qquad \forall\, u\in D(A),\ v\in D(a)=H^2_{\mathrm{rad}}(\Omega).
\]
Its domain is given by
\[
D(A):=\Big\{u\in H^2_{\mathrm{rad}}(\Omega): \exists f\in L^2_{\mathrm{rad}}(\Omega)
\text{ such that } a(u,v)=\langle f,v\rangle_{L^2(\Omega)}\ \forall v\in H^2_{\mathrm{rad}}(\Omega)\Big\},
\]
and we then define $Au:=f$.

The phrase ``free-edge boundary conditions in the trace sense'' means that the boundary expressions produced by integration by parts vanish in the sense of Sobolev traces for all test functions $v\in H^2_{\mathrm{rad}}(\Omega)$.
\begin{claim}
    Let $u\in \Hkrad{4} = \Hk{4} \cap \Lrad$  satisfy the free-edge conditions in the trace sense at $r=R$. Then
    \[
        a(u,v)=\int_\Omega (\Delta^2 u)\,v\,d\textbf{x} \quad \forall v \in \Hkrad{2}.
    \]
\end{claim}
\begin{proof}
    To show this, we turn our attention back to the radial coordinates $r = |\textbf{x}|$ for axisymmetric $u(r),v(r)$.
    One checks that
    \[
        a(u,v)
        =2\pi\int_0^R
        \Big[(1-\nu_p)\Big(u_{rr}v_{rr} + \frac{1}{r^2}u_r v_r\Big)
        + \nu_p\Delta u \Delta v\Big]
        \,r\,dr.
    \]
    Integrating by parts twice, we obtain
    \[
        a(u,v)
        =
        \int_\Omega (\Delta^2 u)\,v\,dx
        \;+\;2\pi\Big[\, r\,M_r(u)\,v_r \;-\; r\,Q_r(u)\,v \,\Big]_{r=0}^{r=R}.
    \]
    where $M_r(u)$ and $Q_r(u)$ are the boundary expressions arising from the
integration by parts. For regular axisymmetric $u,v$ the $r=0$ contribution is $0$ (because of the factor $r$). Hence, if $u$ satisfies the free-edge conditions at $r=R$,
    \[
        M_r(u)(R)=u_{rr}(R)+\nu_p\frac{1}{R}u_r(R)=0,
        \qquad
        Q_r(u)(R)=\frac{d}{dr}\Big(u_{rr}+\frac{1}{r}u_r\Big)\Big|_{r=R}=0,
    \]
    then the boundary term vanishes and, therefore
    \[
        a(u,v)=\int_\Omega (\Delta^2 u)\,v\,dx
        \quad\text{for all  }v\in \Hkrad{2}.
    \]
\end{proof}
\begin{remark}[Identification of the operator domain]
On the disk, standard elliptic regularity for the biharmonic operator applies.
If $u\in D(A)$, then by definition there exists $f\in L^2_{\mathrm{rad}}(\Omega)$ such that $
a(u,v)=\langle f,v\rangle_{L^2(\Omega)}$ for all $v\in H^2_{\mathrm{rad}}(\Omega),$
which implies that $u$ is a weak solution of the biharmonic equation
$\Delta^2 u = f$ with the free-edge boundary conditions.
Since $\Omega$ is smooth and the biharmonic operator is strongly elliptic,
the classical elliptic regularity theory yields
$u\in H^4_{\mathrm{rad}}(\Omega)$ and $Au=\Delta^2 u$. Conversely, if $u\in H^4_{\mathrm{rad}}(\Omega)$ satisfies the free-edge boundary conditions, then integration-by-parts identity yields $a(u,v)=\langle\Delta^2u,v\rangle_{L^2(\Omega)}$ for all $v\in H^2_{\mathrm{rad}}(\Omega),$ so $u\in D(A)$ and $Au=\Delta^2u$. Consequently, on $\Omega$ we may identify $$D(A)=\Big\{
u\in H^4_{\mathrm{rad}}(\Omega):
\Delta^2u\in L^2_{\mathrm{rad}}(\Omega)
\text{ and the free-edge boundary conditions hold}
\Big\}.$$
\end{remark}
Our goal at this point is to show that $A$ has compact resolvent. To this end, for $f\in L^2_{\mathrm{rad}}(\Omega)$ consider the variational problem: find $u \in \Hkrad{2}$ such that
\begin{equation}
    \label{v_prob}
    a(u,v) + \langle u,v \rangle_{L^2(\Omega)} = \langle f,v \rangle_{L^2(\Omega)}, \quad \forall v \in H^2_{\mathrm{rad}}(\Omega).
\end{equation}

\begin{claim}
    The shifted form $(\cdot, \cdot) \mapsto a(\cdot,\cdot) + \langle\cdot,\cdot\rangle_{L^2(\Omega)}$ is bounded and coercive on
$H^2_{\mathrm{rad}}(\Omega)$. 
\end{claim}
\begin{proof}
The boundedness follows from the boundedness of $a(\cdot,\cdot)$ on $H^2(\Omega)$ and Cauchy-Schwarz for the
inner product $L^2$. 
    \begin{align}
        a(u,u) + \langle u,u \rangle_{L^2(\Omega)} &= \|u\|_{L^2(\Omega)}^2 + \int_\Omega \left( (1-\nu_p)(u_{xx}^2 + 2u_{xy}^2 + u_{yy}^2) + \nu_p (\Delta u)^2 \right) d\textbf{x} \\
        &\geq \|u\|_{L^2(\Omega)}^2 + (1-\nu_p) \|\mathcal{D}^2 u\|_{L^2(\Omega)}^2 \geq \alpha \|u\|_{H^2(\Omega)}^2
    \end{align}
    The final step is due to (3.12).
\end{proof}

The shifted form $(\cdot, \cdot) \mapsto a(\cdot,\cdot) + \langle\cdot,\cdot\rangle$ is continuous and coercive, hence, by the Lax-Milgram theorem, ~\eqref{v_prob} is well-posed, that is,
\[
    \forall \ f \in L_{rad}^2(\Omega) \quad \exists ! \ u \in \Hkrad{2} \ \text{ solving } ~\eqref{v_prob}.
\]
Moreover, there exists a $c>0$ such that
\begin{equation}
    \label{l-m_stability}
    \|u\|_{H^2(\Omega)} \leq c \|f\|_{L^2(\Omega)}.
\end{equation} holds for all $f \in L^2_\mathrm{rad}(\Omega)$.

Define the solution operator $B: \Lrad \to \Hkrad{2}$ by
\[
Bf = u,
\]
where $u$ is the unique solution of ~\eqref{v_prob}. Estimate \eqref{l-m_stability} implies that $B$ is bounded.

Now let $u\in D(A)$. Then $$a(u,v) = \langle Au,v\rangle_{L^2(\Omega)}
\qquad \forall v\in H^2_{\mathrm{rad}}(\Omega).$$
Substituting this into ~\eqref{v_prob} yields
\begin{equation}
    \Big\langle (A + I) u, v \Big\rangle_{L^2(\Omega)} = \langle f, v \rangle_{L^2(\Omega)} \quad \forall v \in \Hkrad{2}
\end{equation}
where $I$ is the identity map. Since $\Hkrad{2}$ is dense in $\Lrad$, the above identity implies (see Conway~\cite[Cor.~2.11]{Conway})
\begin{align*}
    (A + I) u &= f \in \Lrad
\end{align*}
Thus, for all $f\in L^2_{\mathrm{rad}}(\Omega),$ $$u = (A+I)^{-1} f.$$ Comparing with $u=Bf$, we obtain the resolvent representation
$$(A+I)^{-1} = \jmath \circ B$$
where $\jmath$ is the \emph{compact} inclusion map (This comes from the Rellich-Kondrachov theorem, see Adams and Fournier or Renardy and Rogers) \cite[Thm.~7.26]{Renardy_Rogers}
\begin{equation}
    \jmath : \Hkrad{2} \hookrightarrow \Lrad
\end{equation}

Since $A$ is self-adjoint, $(A+I)$ is self-adjoint on $L^2_{\mathrm{rad}}(\Omega)$. Moreover, for any bijective self-adjoint operator $T$ one has $(T^{-1})^* = (T^*)^{-1} = T^{-1}$, hence $(A+I)^{-1}$ is again self-adjoint. Combining this with the representation $(A+I)^{-1} = j\circ B$, composition of a compact operator and a bounded operator, we conclude that $(A+I)^{-1}$ is a compact, self-adjoint operator on $L^2_{\mathrm{rad}}(\Omega)$. Consequently, by the spectral theorem for compact self-adjoint operators
(e.g.~\cite[Thm.~8.94]{Renardy_Rogers}), there exists an orthonormal basis
$\{\varphi_n\}_{n=1}^\infty$ of $L^2_{\mathrm{rad}}(\Omega)$ and real eigenvalues $\mu_n\to 0$
such that
\begin{align}
    &(A+I)^{-1}\varphi_n = \mu_n\varphi_n\\
    \implies& (A + I)\varphi_n= \frac{1}{\mu_n}\varphi_n \\
    \implies& A \varphi_n = \left(\frac{1}{\mu_n} - 1\right)\varphi_n
\end{align}
We can now see that $\varphi_n$ also solves the eigenvalue problem presented in section 3.1, with eigenvalues
\[
\lambda_n^4 = \frac{1}{\mu_n} - 1 \quad \text{ and } \quad \phi_n = \varphi_n.
\]
Therefore, $\{\phi_n\}$ forms an orthonormal eigenbasis for $A$ (and hence for the free-edge radial biharmonic problem).

\subsection{Hankel Transforms and Asymptotics}

Throughout the remainder of the paper, we will frequently require the finite Hankel transform of the basis functions
$\phi_{n}(r) = A_{1n} J_{0}\left(\lambda_{n} r\right) + A_{2n} I_{0}\left(\lambda_{n} r\right)$ on $[0,R]$,
where the finite Hankel transform is defined by 
\[
(\mathscr{H}_{[0,R]} f)(\xi)
:= \int_{0}^{R} f(r)\,J_{0}(\xi r)\,r\,dr .
\]
By linearity,
\[
\mathscr{H}_{[0, R]} \phi_{n} = A_{1n} \mathscr{H}_{[0, R]} [J_{0}\left(\lambda_{n} r\right)](\xi) + A_{2n} \mathscr{H}_{[0, R]} [I_{0}\left(\lambda_{n} r\right)](\xi)
\]
Because these are finite Hankel transforms of Bessel functions, closed-form expressions may be obtained using standard Bessel integral identities~\footnote{These Hankel transform identities are verified using a computer algebra system (SymPy). The code used to verify these transforms can be provided upon request.}.
\[
\mathscr{H}_{[0, R]} [J_{0}\left(\lambda_{n} r\right)](\xi) = \begin{cases}
\frac{R \lambda_{n} J_{0}\left(R \xi\right) J_{1}\left(R \lambda_{n}\right)}{\lambda_{n}^{2} - \xi^{2}} - \frac{R \xi J_{0}\left(R \lambda_{n}\right) J_{1}\left(R \xi\right)}{\lambda_{n}^{2} - \xi^{2}} & \text{for}\: \lambda_{n} \neq \xi \\
\frac{R^{2} J^{2}_{0}\left(R \xi\right)}{2} + \frac{R^{2} J^{2}_{1}\left(R \xi\right)}{2} & \text{otherwise}
\end{cases}
\]

\[
\mathscr{H}_{[0, R]} [I_{0}\left(\lambda_{n} r\right)](\xi) = \frac{R \lambda_{n} I_{1}\left(R \lambda_{n}\right) J_{0}\left(R \xi\right)}{\lambda_{n}^{2} + \xi^{2}} + \frac{R \xi I_{0}\left(R \lambda_{n}\right) J_{1}\left(R \xi\right)}{\lambda_{n}^{2} + \xi^{2}}
\]
We will also make use of the large-$\xi$ asymptotic behavior of these transforms
when establishing convergence of certain improper integrals.
Using the well-known decay of Bessel functions,
\[
J_{\nu}(x)=\mathcal O(x^{-1/2})
\qquad (\nu\ge 0,\ x\to\infty),
\]
the above closed-form expressions imply that, for each fixed $n$,
\[
(\mathscr{H}_{[0,R]}\phi_{n})(\xi)
=
\mathcal O(\xi^{-3/2}),
\qquad
\xi\to+\infty,\ \xi\in\mathbb R.
\]

For the same purpose, we also require the large-$\xi$ behavior of the
half-space admittance $\alpha_{\mathrm{HS}}(\xi,\omega)$ defined in
\eqref{eq2}-\eqref{eq3}.
As $\xi\to\infty$ along the real axis,
\[
\alpha(\xi,\omega)=\sqrt{\xi^{2}-k_{L}^{2}}\sim \xi,
\qquad
\sqrt{\xi^{2}-k_{T}^{2}}\sim \xi,
\]
A straightforward large-$\xi$ expansion shows that the $\xi^{4}$ terms
in $\Omega(\xi,\omega)$ cancel identically, and the Rayleigh denominator
satisfies
\[
\Omega(\xi,\omega)
=
2\bigl(k_{L}^{2}-k_{T}^{2}\bigr)\xi^{2}
+
\mathcal O(1).
\]
Consequently,
\[
\alpha_{\mathrm{HS}}(\xi,\omega)
\sim
-\frac{k_{T}^{2}}{2\mu\,(k_{L}^{2}-k_{T}^{2})}\,\xi^{-1},
\qquad
\xi\to\infty,
\]
and in particular
\[
\lim_{\xi\to\infty}\xi\,\alpha_{\mathrm{HS}}(\xi,\omega)
=
-\frac{k_{T}^{2}}{2\mu\,(k_{L}^{2}-k_{T}^{2})},
\qquad
\alpha_{\mathrm{HS}}(\xi,\omega)=\mathcal O(\xi^{-1}).
\]

\section{Modal Representation of the Truncated Half-Space Operator}

In the infinite-domain setting, Lamb's half-space relation is diagonal in the Hankel variable $\xi$,
the surface displacement and traction satisfy a multiplier law of the form
$\widehat{w}^{\,0}(\xi,\omega)=\alpha_{HS}(\xi,\omega)\widehat{q}^{\,0}(\xi,\omega)$.
For a plate of finite radius $R$, however, the contact restriction $r\in[0,R]$ breaks Hankel diagonality:
although $\widehat{q}(\cdot,\omega)$ is supported on $[0,R]$, the corresponding surface displacement obtained from the infinite-domain half-space operator
is nonlocal and, prior to restriction, is not supported on $[0,R]$. Consequently, to couple the half-space
to the spectral discretization of the disk in Section~5, we work with the physically relevant restricted operator
\[
\mathcal{M}(\omega)\,\widehat{q}(r,\omega)
:=\chi_{[0,R]}(r)\,\mathscr{H}_0^{-1}\!\Big[\alpha_{HS}(\xi,\omega)\,\mathscr{H}_{[0, R]}[\widehat{q}(\cdot,\omega)](\xi)\Big](r),
\qquad r\in[0,R],
\]
and seek (i) an explicit matrix representation of $\mathcal{M}(\omega)$ on the finite-disk basis $\{\phi_n\}$, and
(ii) an efficient procedure for applying its inverse on the same basis.

The central object in this section is the matrix $S(\omega + i\eta)$ associated with the regularized
operator $\mathcal{M}_\eta(\omega)$, obtained by introducing the small imaginary shift $\omega\mapsto \omega+i\eta$
with $\eta>0$ to control the real-axis singularities of $\alpha_{HS}(\xi,\omega)$. In Section~4.1 we derive the matrix elements $\widehat S_{km}(\omega + i\eta)$ and show how the response associated with a given basis function $\phi_n$ is
obtained by solving the corresponding linear system on the disk basis.

The remaining subsections explain the analytic structure that makes these matrix entries interesting.
The half space admittance (Neumann-to-Dirichlet symbol) $\alpha_{HS}$ inherits square-root branch points at the longitudinal and shear thresholds and,
most importantly, a real-axis pole at $\xi=\xi_R(\omega)$ corresponding to the possibility of Rayleigh waves on the surface of the half-space. Section~4.2 reviews the Rayleigh wave equation $\Omega(\xi,\omega)=0$ and
connects it directly to the simple poles that appear inside the $\xi$-integrals defining
$\widehat{S}_{km}(\omega + i\eta)$. Section~4.3 then introduces the Cauchy principal value interpretation of the limiting
integrals as $\eta\to0^+$, while Section~4.4 isolates the pole contribution (handled via
contour deformation/residue terms) that will later be interpreted as radiation damping into the half-space.

  \subsection{Matrix Elements of \texorpdfstring{$\mathcal{M}(\omega)$}{M(ω)}}
  \label{sec:M}
The infinite-domain Lamb relation (see, e.g., Chen et al.~\cite{Chen1988} as a summary of Lamb's solution)
gives the relationship between surface traction and surface displacement in the
Hankel-frequency domain:
\begin{equation*}
    \widehat{u_2}^{\,0}(\xi, \omega) = \alpha_{HS}(\xi, \omega) \widehat{q}^{\,0}(\xi, \omega)
\end{equation*}
where $u_2^0$ is the zeroth order Hankel transform of the vertical displacement at the surface of the soil ($z = 0, r\in [0, \infty)$), and $q^0$ is the zeroth order Hankel transform of the traction at the surface.
On the contact region $r\in[0,R]$ we impose the coupling condition $\widehat u_2(r,\omega)=\widehat w(r,\omega)$, but $\widehat u_2$ is a priori unknown for $r>R$ due to the nonlocality of the Hankel transform.
In order to solve our system with a spectral expansion, we desire an expression for $q$ in terms of $w$.
To reveal $w$ we take the inverse Hankel transform of both sides and restrict to $r \in [0, R]$:
\begin{equation*}
    \widehat{w}(r, \omega) = \chi_{[0, R]}(r) \mathscr{H}_0^{-1} \left[\alpha_{HS}(\xi, \omega) \widehat{q}^{\,0}(\xi, \omega)\right]
\end{equation*}
Since $q$ is only supported on $[0, R]$, we can write
\begin{equation*}
    \widehat{q}^{\,0}(\xi, \omega) = \mathscr{H}_{[0, R]}[\widehat{q}(r)] = \int_0^R \widehat{q}(r, \omega) J_0(\xi r) r dr
\end{equation*}
Thus, we define the restricted half-space operator
\begin{equation*}
    \mathcal{M} \widehat{q}(r, \omega) := \chi_{[0, R]}(r) \mathscr{H}_0^{-1} \left[\alpha_{HS}(\xi, \omega) \mathscr{H}_{[0, R]}[\widehat{q}(r, \omega)]\right]
\end{equation*}
$\alpha_{HS}(\xi, \omega)$ has poles on the real $\xi$ axis, so we shift them off the axis by adding a small positive imaginary part to the frequency $\omega + i \eta$ with $\eta > 0$, and we eventually take the limit $\eta \to 0^+$

\begin{equation*}
    \mathcal{M}_{\eta} \widehat{q}(r, \omega) := \chi_{[0, R]}(r) \mathscr{H}_0^{-1} \left[\alpha_{HS}(\xi, \omega + i \eta) \mathscr{H}_{[0, R]}[\widehat{q}(r, \omega)]\right]
\end{equation*}

Let $\{\phi_n\}_{n\ge 1}$ be a basis of $L^2_r([0,R])$ with inner product
$\langle f,g\rangle_{L^2_r([0,R])}:=\int_0^R f(r)g(r)\,r\,dr$.
For use in our spectral expansion, we seek $\psi_n$ such that
\begin{equation}
\label{eq:inv}
    \mathcal{M}_{\eta} \psi_n(r) = \phi_n(r) \quad r \in [0, R]
\end{equation}
We expand $\psi_n$ on the basis:
\begin{equation*}
    \psi_n(r) = \sum_{k} C_{nk} \phi_k(r) \quad r \in [0, R]
\end{equation*} and
substitute this sum into~\eqref{eq:inv}
\begin{equation*}
    \phi_n(r) = \chi_{[0, R]}(r) \mathscr{H}_0^{-1} \left[\alpha_{HS}(\xi, \omega + i \eta) \mathscr{H}_{[0, R]} \left[\sum_{k} C_{nk} \phi_k(r)\right]\right].
\end{equation*}
To use orthogonality of the basis, we multiply both sides by $\phi_m(r) r$ and integrate over $[0, R]$:
\begin{equation}
    \int_0^R \phi_n(r) \phi_m(r) r dr = \sum_{k} C_{nk} \int_0^R \phi_m(r) \mathscr{H}_0^{-1} \left[\alpha_{HS}(\xi, \omega + i \eta) \mathscr{H}_{[0, R]}[\phi_k(r)]\right] r dr
\end{equation}
Define the matrices,
\begin{align}
    N_{nm} &= \int_0^R \phi_n(r) \phi_m(r) r dr \\
    \widehat{S}_{km}(\omega + i\eta) &= \int_0^R \phi_m(r) \mathscr{H}_0^{-1} \left[\alpha_{HS}(\xi, \omega + i \eta) \mathscr{H}_{[0, R]}[\phi_k(r)]\right] r dr.
\end{align}
Then the coefficients satisfy the dense linear system
\begin{equation}
\label{eq:sys_S}
    \mathbf{C}\,\boldsymbol{\widehat{S}}(\omega+i\eta)=\mathbf{N},
\end{equation}
which we solve for $\mathbf C$ (matrix with entries $C_{nk}$) (numerically, e.g.\ via LU/QR), and then recover
\[
\psi_n(r)=\sum_k C_{nk}\phi_k(r) \quad r \in [0, R].
\]
The integral expression we have for the matrix $\boldsymbol{\widehat{S}}(\omega + i\eta)$ is not convenient for computation because part of the integrand is in the spatial domain, and part is in the Hankel domain, so numerically we would have to account for two separate grids. Moreover, in this form, we would have to approximate the inverse Hankel transform.
However, if we define the zero-extension of the basis functions:
\begin{equation*}
    \widetilde{\phi}_n(r) = \begin{cases}
        \phi_n(r) & r \in [0, R] \\
        0 & r \in (R, \infty)
    \end{cases}
\end{equation*}
we can use Plancherel's theorem for the Hankel transform to rewrite the integral in terms of only Hankel transforms:
\begin{align}
    \widehat{S}_{km}(\omega + i\eta) &= \int_0^R \phi_m(r) \mathscr{H}_0^{-1} \left[\alpha_{HS}(\xi, \omega + i \eta) \mathscr{H}_{[0, R]}[\phi_k(r)]\right] r dr \\
    &= \int_0^{\infty} \widetilde{\phi}_m(r) \mathscr{H}_0^{-1} \left[\alpha_{HS}(\xi, \omega + i \eta) \mathscr{H}_0[\widetilde{\phi}_k(r)]\right] r dr \\
    &= \int_0^{\infty} \alpha_{HS}(\xi, \omega + i \eta) \mathscr{H}_{[0, R]}[\phi_k(r)](\xi) \mathscr{H}_{[0, R]}[\phi_m(r)](\xi) \xi d\xi
\end{align}
With this simplification, we do not have to approximate any Hankel transforms numerically because we already have exact expressions for $\mathscr{H}_{[0, R]}[\phi_n(r)]$.

  \subsection{Rayleigh Denominator and Singular Structure}
 
The Hankel domain multiplier that defines the soil operator $\alpha_{HS}(\xi, \omega) = -\frac{\alpha(\xi, \omega)k_T^2(\omega)}{\mu \Omega(\xi, \omega)}$ 
blows up at the Rayleigh wave number $\xi=\xi_R(\omega) \in \R$ for which the multiplier's denominator $\Omega(\xi,\omega)$ vanishes while the numerator remains finite and nonzero. Physically, Rayleigh waves are a combination of longitudinal and shear waves and propagate along the free surface of the half-space. See also the definition of $\Omega$ (1.3), the square-root factors $\sqrt{\xi^2-k_T^2}$ and $\sqrt{\xi^2-k_L^2}$ introduce branch points at $\xi=k_T$ and $\xi=k_L$ where the $\xi$-derivative of $\alpha_{HS}$ blows up. In the axisymmetric setting, we restrict ourselves to $\xi\ge 0$; only positive branch points and poles are encountered. Physically, longitudinal waves travel faster than shear waves, we take $k_L(\omega) < k_T(\omega)$ for the remainder of our analysis.

The Rayleigh wavenumber satisfies the dispersion relation $\Omega(\xi,\omega)=0$, written equivalently as
\[
\left(2\xi^2 - \frac{\omega^2}{C_T^2}\right)^2
= 4 \xi^2 \left(\xi^2 - \frac{\omega^2}{C_T^2}\right)^{1/2}
               \left(\xi^2 - \frac{\omega^2}{C_L^2}\right)^{1/2},
\]
which is the classical Rayleigh wave equation.  Dividing by $\xi^4$ and introducing
the nondimensional phase velocity $\zeta=\left(\frac{\omega}{\xi C_T}\right)^2$ and the wave-speed ratio
$\kappa=\frac{C_T}{C_L}\in(0,1)$ gives
\begin{equation}\label{eq:dimension-less}
    (2-\zeta)^2 = 4(1-\zeta)^{1/2}\,(1 - \kappa^2 \zeta)^{1/2}.
\end{equation}
See~\cite[\S24]{landau7} for more a more detailed treatment of this relation.
For each physically relevant $\kappa\in(0,1)$ there is a unique solution $\zeta\in(0,1)$, and the
corresponding Rayleigh wave speed is $c_R=\sqrt{\zeta}\, C_T<C_T$.\footnote{In the literature on elasticity, it is common to rationalize \ref{eq:dimension-less} to a cubic polynomial in $\zeta$ by squaring both sides, which introduces spurious zeros that cannot occur in the original equation; We prefer to work with the original form \ref{eq:dimension-less} to avoid this complication.}
Similar to the transverse and shear wave speeds, the Rayleigh wave number is defined as
\[
\xi_R(\omega)=\frac{\omega}{c_R}
\]
which lies to the right of the shear and longitudinal branch points:
\[
\xi_R(\omega) > k_T(\omega)=\frac{\omega}{C_T} > k_L(\omega)=\frac{\omega}{C_L}.
\]
In the next sections, we give a precise meaning to the singular integral
(via the principal value) that we must compute to assemble the $\boldsymbol{\widehat{S}}$ matrix.

  \subsection{Cauchy Principal Values}
  
The matrix introduced in Section~4.1 is defined, for $\eta>0$, by the integral formula
\[
\widehat{S}_{km}(\omega + i\eta)
=\int_{0}^{\infty}\alpha_{HS}(\xi,\omega+i\eta)\,
\mathscr{H}_{[0,R]} \phi_k\,
\mathscr{H}_{[0,R]} \phi_m\,\xi\,d\xi .
\]

\begin{remark}[Small shift in frequency]
We adopt the standard limiting-absorption prescription $z=\omega+i\eta$ with $\eta>0$, for which the integral is well defined and convergent, and define the physical quantity by $\widehat{S}_{km}(\omega+0^+):=\lim_{\eta\to0^+}\widehat{S}_{km}(\omega+i\eta)$. 
With our sign convention, $\widehat{f}(\omega) = \int_{-\infty}^{\infty} f(t)\,e^{i\omega t}\,dt,$ the frequency shift $z=\omega + i\eta$ introduces a factor $e^{-\eta t}$ in the Fourier transform of $S_{km}(t)$. This guarantees convergence of the representation 
\[
\widehat S_{km}(z)=\int_{0}^{\infty} S_{km}(t)\,e^{izt}\,dt,\qquad \Im z>0,
\]
and enforces causality in time: $S_{km}(t)=0$ for $t<0$ if and only if its Fourier transform is analytic in the upper half-plane, and admits boundary values $\widehat S_{km}(\omega+i0^+)$
on $\mathbb R$ (\cite[\S3]{GreensFunctionCausality}). Causality therefore requires $\widehat S_{km}(\omega)$ to be analytic for $\Im \omega>0$, so that closing the inverse-transform contour in the upper half-plane for $t<0$ encloses no singularities and gives $S_{km}(t)=0$.
By contrast, taking the opposite boundary value $\widehat S_{km}(\omega-i0^+)$ corresponds to
an analytic continuation from the \emph{lower} half-plane and yields the acausal kernel supported on $t\le 0$ (time-reversed response).
\end{remark}

In order to give a precise meaning to the real-frequency limit ($\eta \to 0^+$), we can use the Sokhotski-Plemelj formula to show that the limit is equal to the principal value integral of $\alpha_{HS}\mathscr{H}_{[0,R]} \phi_k\,\mathscr{H}_{[0,R]} \phi_m\,$ plus the residue contribution.
Then use contour integration to show that the limit is well-defined.
\begin{claim}[Boundary value]\label{prop:Skm_PV}
For fixed $\omega>0$, $\widehat S_{km}(\omega+i\eta)$ is finite for every $\eta>0$ and admits a limit
\[
\widehat S_{km}(\omega+0^+):=\lim_{\eta \to 0+}\widehat S_{km}(\omega+i\eta).
\]
Moreover,
\begin{equation}\label{eq:PV_res}
\widehat S_{km}(\omega+0^+) = \PV\int_{0}^{\infty}\alpha_{HS}(\xi,\omega)\, \phi_k^0(\xi) \phi_m^0(\xi)\,\xi\,d\xi + i\pi \ \res{\xi}{\xi_R(\omega)} \left(\alpha_{HS}(\xi,\omega)\right)\, \phi_k^0(\xi_R)\phi_m^0(\xi_R)\,\xi_R,
\end{equation}
and the right-hand side is finite.
\end{claim}

\begin{proof}
First, we show the equality \eqref{eq:PV_res} using a Laurent expansion and the Sokhotski-Plemelj formula, then we show the finiteness of the right-hand side via contour integration.
The pole at $\xi = \xi_R$ is of order one ($\partial_{\xi} \Omega(\xi_R, \omega) \neq 0$) and the Hankel transformed factors are entire.

We can write $\alpha_{HS}(\xi, \omega)$ as a Laurent expansion around the pole at $\xi = \xi_R$:
\begin{equation}
    \alpha_{HS}(\xi, \omega) = \frac{A_R(\omega)}{\xi - \xi_R} + \sum_{n=0}^{\infty} B_n(\omega) (\xi - \xi_R)^n = \frac{A_R(\omega)}{\xi - \xi_R} + B(\xi, \omega)
\end{equation}
Where $A_R(\omega)$ is the residue of $\alpha_{HS}$ at $\xi = \xi_R$.
\[
A_R(\omega) = \res{\xi}{\xi_R} \alpha_{HS} =
 -\frac{\alpha(\xi_R, \omega) k_T^2(\omega)}{\mu \, [\partial_{\xi} \Omega(\xi, \omega)]_{\xi = \xi_R}}
\]
$B(\xi, \omega)$ inherits the square-root cusps but is analytic for $\xi > K_T$. Because $\alpha_{HS}(\xi, \omega) = \mathcal{O}(\xi^{-1})$ as $\xi \to \infty$, $B(\xi, \omega) = \mathcal{O}(\xi^{-1})$ as $\xi \to \infty$ as well.
For brevity, we define $g_{km}(\xi):= \mathscr{H}_{[0,R]} \phi_k(\xi) \,\mathscr{H}_{[0,R]} \phi_m(\xi)$ = $\mathcal{O}(\xi^{-3})$
Then we rewrite $S_{km}(\omega + i \eta)$ as
\begin{align}
    \widehat{S}^{}_{km}(\omega + i\eta) = &\int_{0}^{k_T} \alpha_{HS}(\xi, \omega + i \eta) \ g_{km}(\xi) \ \xi \ d\xi \\
    + A_R(\omega + i \eta) \ &\int_{k_T}^{\infty} \frac{g_{km}(\xi) \ \xi}{\xi - \xi_R(\omega + i \eta)} \ d\xi \\
    + &\int_{k_T}^{\infty} B(\xi, \omega + i \eta) \ g_{km}(\xi) \ \xi \ d\xi
\end{align}

Let us focus on the second integral. We can use the definition of $\xi_R(\omega + i \eta)$ to rewrite the denominator:
\begin{equation}
    \xi - \xi_R(\omega + i \eta) = \xi - \xi_R(\omega) - i \frac{\eta}{c_R}
\end{equation}
So we can use the Sokhotski-Plemelj formula to take the limit $\eta \to 0^+$:
\begin{align}
    \lim_{\eta \to 0^+} &A_R(\omega + i \eta) \ \int_{k_T}^{\infty} \frac{g_{km}(\xi) \ \xi}{\xi - \xi_R(\omega) - i \frac{\eta}{c_R}} \ d\xi \\
    &= A_R(\omega) \PV \int_{k_T}^{\infty} \frac{g_{km}(\xi) \ \xi}{\xi - \xi_R(\omega)} \ d\xi + i \pi A_R(\omega) g_{km}(\xi_R(\omega)) \ \xi_R(\omega)
\end{align}
Thus we have
\begin{align}
    \widehat{S}_{km}(\omega+0^+) &= \int_{0}^{k_T} \alpha_{HS}(\xi, \omega) \ g_{km}(\xi) \ \xi \ d\xi \\
    &+ \PV \int_{k_T}^{\infty} \left( \frac{A_R(\omega)}{\xi - \xi_R(\omega)} +  B(\xi, \omega) \right) g_{km}(\xi) \ \xi \ d\xi \\
    &+ i \pi A_R(\omega) g_{km}(\xi_R(\omega)) \ \xi_R(\omega)
\end{align}
or equivalently,
\begin{align}
    \widehat{S}_{km}(\omega+0^+) &= \PV \int_{0}^{\infty} \alpha_{HS}(\xi, \omega) \ g_{km}(\xi) \ \xi \ d\xi + i \pi A_R(\omega) g_{km}(\xi_R(\omega)) \ \xi_R(\omega)
\end{align}

\paragraph{Boundedness of the Limit}
Next, we show the boundedness of the right-hand side of \eqref{eq:PV_res} through contour integration.
Consider the \emph{counterclockwise} contour 
\[
\Gamma = [0, k_L-\varepsilon] \cup [k_L + \varepsilon, k_T-\varepsilon] \cup [k_T + \varepsilon, \xi_R - \varepsilon] \cup [\xi_R + \varepsilon, \xi_{\text{max}}] \cup \left(\bigcup_{i=1}^{3} \gamma_i \right) \cup \left(\bigcup_{i=1}^{2} B_i \right) \cup \sigma
\]
as shown in figure \ref{rectangle}, where $\gamma_1$ is the top horizontal segment from $(\xi_{\text{max}}, y)$ to $(0, y)$, $\gamma_2$ is the left vertical segment from $(0, y)$ to $(0, 0)$, $\gamma_3$ is the right vertical segment from $(\xi_{\text{max}}, 0)$ to $(\xi_{\text{max}}, y)$, $B_1$ and $B_2$ are the \emph{clockwise} semicircular indentations around $k_L$ and $k_T$ in the upper half-plane, respectively, and $\sigma$ is the \emph{clockwise} semicircular indentation around $\xi_R$ in the upper half-plane.

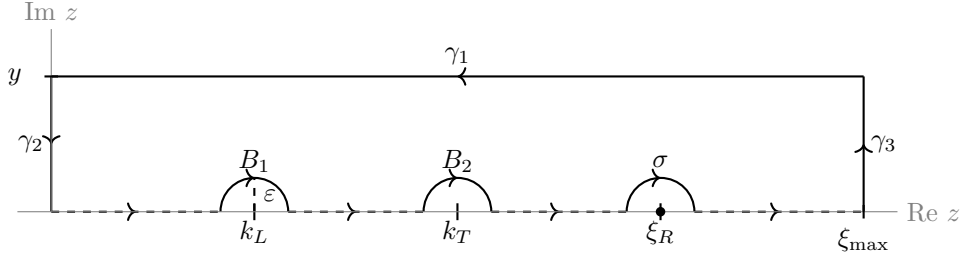
\begin{figure}[htbp]
\centering
\begin{tikzpicture}[scale=0.9]
  % Parameters
  \def\Xmin{-6}
  \def\Xmax{6}
  \def\Ymax{2}
  \def\rA{0.5}
  \def\rB{0.5}
  \def\rC{0.5}
  \def\xA{-3}
  \def\xB{0}
  \def\xC{3}

  % Rectangle boundary (excluding indentations)
  % Bottom segments with gaps for semicircle indentations
  \draw[thick, ->--, black] (\Xmin,0) -- (\xA-\rA,0);
  \draw[thick, ->--, black] (\xA+\rA,0) -- (\xB-\rB,0);
  \draw[thick, ->--, black] (\xB+\rB,0) -- (\xC-\rC,0);
  \draw[thick, ->--, black] (\xC+\rC,0) -- (\Xmax,0);
  % Right vertical
  \draw[thick,->-] (\Xmax,0) -- (\Xmax,\Ymax);
  % Top
  \draw[thick,->-] (\Xmax,\Ymax) -- (\Xmin,\Ymax);
  % Left vertical
  \draw[thick,->-] (\Xmin,\Ymax) -- (\Xmin,0);

  % Indentation semicircles (upper half-plane, clockwise)
  \draw[thick,->-] (\xA-\rA,0) arc (180:0:\rA);
  \draw[thick,->-] (\xB-\rB,0) arc (180:0:\rB);
  \draw[thick,->-] (\xC-\rC,0) arc (180:0:\rC);

  % Dashed line to indicate radius of indentations
  \draw[thick, dashed] (\xA,0) -- (\xA,\rA);
  % Label for epsilon
  \node[right] at (\xA,\rA/2) {$\varepsilon$};

  % Small dash on the real axis at the center of each indentation
  \draw[thick] (\xA,0.125) -- (\xA,-0.125);
  \draw[thick] (\xB,0.125) -- (\xB,-0.125);
  \draw[thick] (\xC,0.125) -- (\xC,-0.125);

  % Small circles around upper poles (counterclockwise, included)
%   \draw[thick,-<-] (\poleTwoX,\poleTwoY) circle [radius=0.35];

  % Axes
  \draw[-,gray] (-6.5,0) -- (6.5,0) node[right] {$\text{Re} \ z$};
  \draw[-,gray] (\Xmin,0) -- (\Xmin,\Ymax+0.7) node[above] {$\text{Im} \ z$};

  % Labels for indentations
  \node[below] at (\xA,0) {$k_L$};
  \node[below] at (\xB,0) {$k_T$};
  \node[below] at (\xC,0) {$\xi_R$};

  % Label for indentation semicircles
    \node at (\xA,0.75) {$B_1$};
    \node at (\xB,0.75) {$B_2$};
    \node at (\xC,0.75) {$\sigma$};

  % Dot at the pole on the real axis
  \fill (\xC,0) circle (2pt);

  % Pole markers
%   \fill (\poleTwoX,\poleTwoY) circle (1pt);
%   \node[right] at (\poleTwoX+0.3,\poleTwoY) {$\sigma_3$};

  % Contour label
  \node at (0,\Ymax+0.3) {$\gamma_1$};
  \node at (\Xmin-0.3,\Ymax/2) {$\gamma_2$};
  \node at (\Xmax+0.3,\Ymax/2) {$\gamma_3$};

  %small dash on the imaginary axis for ymax
  \draw[thick] (\Xmin+0.1,\Ymax) -- (\Xmin-0.1,\Ymax);
  %Label for ymax
  \node[left] at (\Xmin-0.3,\Ymax) {$y$};

  %Small dash on the Real axis for R
    \draw[thick] (\Xmax,0.1) -- (\Xmax,-0.1);
    %Label for R
    \node[below] at (\Xmax,-0.1) {$\xi_\text{max}$};

\end{tikzpicture}
\caption{Contour Path}
\label{rectangle}
\end{figure}

$y$ remains fixed, while $\xi_{\rm{max}} \to \infty$. We fix branches of the square root functions so that
$\sqrt{z^{2}-k_{L}^{2}}$ and $\sqrt{z^{2}-k_{T}^{2}}$ are analytic in the
region $\{0\le \Re z \le \xi_{\max},\, 0\le \Im z \le y\}$ away from the
branch points at $z=k_L$ and $z=k_T$; the branch cuts are chosen to lie
outside this region, e.g. the cuts can be chosen as vertical rays
$\{\xi=k_j - i s: s\ge0\}$, $j\in\{L,T\}$. For $\varepsilon>0$ the integrand
$\alpha_{HS}(z,\omega)\,g_{km}(z)\,z$ is analytic in the region enclosed by
$\Gamma$ except for a simple pole at $z=\xi_R(\omega)$, which is avoided by
the indentation $\sigma$. Therefore, by Cauchy's integral theorem,
\begin{equation}
    0=\int_{\Gamma} \alpha_{HS}(z, \omega) g_{km}(z) z dz = \Big(\text{real-axis segments}\Big)
+ \sum_{i=1}^3\int_{\gamma_i}\cdots
+ \sum_{i=1}^2\int_{B_i}\cdots
+ \int_\sigma \cdots .
\end{equation}
Let $f(z)=\alpha_{HS}(z,\omega)\,g_{km}(z)\,z$. Since $\xi_R$ is a simple pole
of $f$ and $\sigma$ is the clockwise upper semicircle $z=\xi_R+\varepsilon
e^{i\theta}$, $\theta\in[0,\pi]$, we have
$\lim_{\varepsilon \rightarrow 0}\int_\sigma \alpha_{HS}(z,\omega)\,g_{km}(z)\,z\,dz= -\, i\pi\,\text{Res}_{z=\xi_R(\omega)}
\big(\alpha_{HS}(z,\omega)\,g_{km}(z)\,z\big).$
We can now split the contour integral into its components and rearrange for the principal value integral (the segments along the real axis excluding the indentations) plus the residue contribution from the pole at $\xi_R$:
\begin{align}
    \PV &\int_{0}^{\infty} \alpha_{HS}(\xi, \omega) g_{km}(\xi) \xi d\xi = \\
    &-\lim_{\varepsilon \to 0, \xi_{\text{max}} \to \infty} \Bigg[ \sum_{i=1}^2 \int_{B_i} \alpha_{HS}(z, \omega) g_{km}(z) z dz
    + \sum_{i=1}^3 \int_{\gamma_i} \alpha_{HS}(z, \omega) g_{km}(z) z dz
    \Bigg] \\
    &+ i\pi\,\text{Res}_{z=\xi_R(\omega)}
\big(\alpha_{HS}(z,\omega)\,g_{km}(z)\,z\big)
\end{align}
\begin{enumerate}[label=(\roman*)]
    \item \textbf{Top horizontal segment $\gamma_1$}: As $\xi_{\text{max}} \to \infty$, the integrand decays like $\xi_{\text{max}}^{-3}$ and is analytic on the path, so the integral along this segment is finite in the limit.
    \item \textbf{Left vertical segment $\gamma_2$}: For a fixed $y>0$ and $\xi = 0$ the integrand is analytic on the path and the contour is bounded, so the integral along this segment is finite.
    \item \textbf{Right vertical segment $\gamma_3$}: As $\xi_{\text{max}} \to \infty$ the integrand decays like $\xi_{\text{max}}^{-3}$ and $y$ remains constant, so in the limit, the integral along this path is zero.
    \item \textbf{Semicircular indentations $B_1$ and $B_2$}: 
    Near each branch point
$k\in\{k_L,k_T\}$ the integrand has square-root branch points at $k_L$ and $k_T$, that is, it has the local behavior
\[
\alpha_{HS}(\xi,\omega)\,g_{km}(\xi)\,\xi = \mathcal O\big((\xi-k)^{-1/2}\big),
\qquad \xi\to k,
\]
so on the indentation $\xi=k+\varepsilon e^{i\theta}$ ($0\le\theta\le\pi$) we have
\[
\left|\int_{B_k}\alpha_{HS}(\xi,\omega)\,g_{km}(\xi)\,\xi\,d\xi\right|
\le C\int_0^\pi \varepsilon^{-1/2}\,\varepsilon\,d\theta
= \mathcal O(\varepsilon^{1/2})\xrightarrow[\varepsilon\to0]{}0.
\]
Therefore, as $\varepsilon \to 0$, the integrals along these segments are zero.
\end{enumerate}
Although individual contour integrals depend on the chosen height $y>0$, their sum is independent of $y$ by analyticity; consequently, the limiting real-axis representation obtained after $\varepsilon\to0$ and $\xi_{\max}\to\infty$ is unique, well defined and yields the representation \eqref{eq:PV_res}, in which the principal value integral is finite and the contribution of the Rayleigh pole is given by the residue term. Hence $\widehat S_{km}(\omega+0^+)$ is the well-defined physical matrix.
\end{proof}

\begin{remark}[Well-posedness versus numerical approximation]
The use of contour integration along a rectangular path with small indentations at the singular points ensures that \eqref{eq:PV_res} is well defined, produces the matrix $\widehat{S}_{km}(\omega+0^+)$, and relies solely on analytic properties; it does not depend on specific numerical approximations. 

The Bernstein-ellipse framework introduced in section 5.2 serves a different purpose. It is used to describe the region of complex analyticity of the transformed integrands in a manner directly compatible with classical Gauss-Legendre theory, thereby providing quantitative exponential convergence estimates for the numerical quadrature.
\end{remark}

  \subsection{Residue Contributions and Radiation Damping}
The improper integrals for the entries of the $\widehat{\mathbf{S}}$ matrix contain a simple pole on the real $\xi$-axis at the Rayleigh wavenumber $\xi=\xi_R(\omega)$, which corresponds to the zero of the Rayleigh denominator
$\Omega(\xi,\omega)$. This pole represents the surface (Rayleigh) wave of the elastic half-space. As stated in the previous subsection, these integrals must be interpreted in the sense of the limiting absorption principle,
i.e., as limits $\omega\mapsto\omega+i0^+$. This yields a decomposition of the form (4.19). The residue term is purely imaginary and isolates the contribution of the Rayleigh surface wave to the plate-soil interaction. The principal value integral may acquire an imaginary part, reflecting radiation
into the compressional and shear wave continua associated with the branch-cut structure of the Lamb kernel.

The imaginary part of $\widehat{\mathbf{S}}(\omega+0^+)$ has a physical interpretation. It represents radiation of energy from the plate into the underlying half-space as outward-propagating waves, including both body waves associated with the continuous spectrum and the surface Rayleigh wave associated with the pole. This mechanism gives rise to frequency-dependent radiation damping, even in the absence of material
dissipation.

From an operator-theoretic point of view, the appearance of imaginary terms is a direct consequence of the continuous spectrum of the elastic half-space and the enforcement of the outgoing wave condition via the limiting absorption principle. In the frequency domain, the resulting response operator is therefore non-self-adjoint, reflecting the fact that energy is
not conserved within the plate alone but is radiated into the surrounding
half-space.

\section{Assembly and Inversion of the Discrete System}
\subsection{Spectral Truncation and Discrete Formulation}
  Using our basis $\{\phi_n\}$, we discretize the problem outlined in \ref{eq:plate_freq}.
  Let $\widehat{w}(r, \omega) = \sum_n^{\infty} a_n(\omega) \phi_n(r)$.
  \begin{equation}
    \Big(D\,\Delta^2 - \rho h\,\omega^2\Big)\sum_n^N a_n(\omega) \phi_n(r) + \widehat q(r,\omega)
    = \frac{\widehat p(\omega)}{2\pi r}\delta(r), \qquad 0\le r\le R.
  \end{equation}
  Using the inverted soil operator action on $\phi_n$ derived in \ref{sec:M}
  \begin{equation*}
    \widehat{q} = \sum_{n=1}^{\infty} a_n(\omega) \psi_n
  \end{equation*}
Multiply by $\phi_m(r) \ r$ and integrate over $[0, R]$ to use the orthogonality of $\{\phi_n\}$
  \begin{align}
    &D \sum_{n=1}^{\infty} a_n(\omega) \int_{0}^{R} \Delta^2 \phi_n \phi_m r dr - \rho h \omega^2 \sum_{n=1}^{\infty} a_n(\omega) \int_{0}^{R} \phi_n \phi_m r dr \\
    &+ \sum_{n=1}^{\infty} a_n(\omega) \int_{0}^{R} \psi_n \phi_m r dr = \frac{\widehat{p}(\omega)}{2\pi} \phi_m(0)
\end{align}
Define the diagonal normalization matrix $N_{nm} = \int_0^R \phi_n(r) \phi_m(r) \ rdr$ and use the fact that $\Delta^2 \phi_n = \lambda_n^4 \phi_n$ to rewrite the above equation:
\begin{equation}
\label{ep:sys_K}
    \sum_{n=1}^h a_n \left[ (D \lambda_n^4 - \rho h \omega^2) N_{nm} + \int_{0}^{R} \psi_n \phi_m r dr \right] = \frac{\widehat{p}(\omega)}{2\pi} \phi_m(0)
\end{equation}
  \subsection{Evaluation of \texorpdfstring{$\widehat{S}_{km}(\omega)$}{Sₖₘ(ω)}}
Evaluation of the matrix $\mathbf{\widehat{S}}$ requires the computation of the improper integral:
\begin{equation}
    \widehat{S}_{km}(\omega+0^{+}) = \text{P.V.} \int_{0}^{\infty} \alpha_{HS}(\xi,\omega) g_{km}(\xi) \xi \, d\xi + i\pi A_{R}(\omega) g_{km}(\xi_{R}) \xi_{R}.
\end{equation}
The residue term is explicitly evaluated from the Laurent expansion at $\xi=\xi_R$, but the smooth portion of the principal value integral must be approximated. The principal value integrand has a simple pole at $\xi = \xi_R(\omega)$, as well as branch-point singularities at $\xi = k_L$ and $\xi = k_T$. Additionally, the integral is defined over an infinite domain. 
\\
\par First, the integral is truncated to a sufficiently large finite value $\xi_{\text{tail}} > \xi_R$.
The approximation becomes:

\begin{equation}
\widehat{S}_{km}(\omega+0^{+})
=
\text{P.V.}\int_{0}^{\xi_{\text{tail}}}
\alpha_{HS}(\xi,\omega)\,g_{km}(\xi)\,\xi\,d\xi
\;+\;
i\pi A_R(\omega)\,g_{km}(\xi_R)\,\xi_R
\;+\;\tau(\xi_{\text{tail}}),
\end{equation}

\[
\tau(\xi_{\text{tail}})
:=\int_{\xi_{\text{tail}}}^{\infty}
\alpha_{HS}(\xi,\omega)\,g_{km}(\xi)\,\xi\,d\xi.
\]
Since the integrand decays as $\mathcal{O}(\xi^{-3})$ for large $\xi$, the truncation error $\tau(\xi_{\text{tail}})$ introduces an algebraic error bounded by:
\begin{equation}
    |\tau(\xi_{\text{tail}})| \le C \int_{\xi_{\text{tail}}}^{\infty} \xi^{-3} \, d\xi = \mathcal{O}(\xi_{\text{tail}}^{-2}).
\end{equation}This decay follows from the large-$\xi$ asymptotic behavior of the Hankel
transforms appearing in $g_{km}(\xi)$.
Using the well-known decay of Bessel functions,
\[
J_\nu(x)=\mathcal O(x^{-1/2}), \qquad x\to\infty,
\]
we showed earlier that, for each fixed $n$,
\[
(\mathscr H_{[0,R]}\phi_n)(\xi)=\mathcal O(\xi^{-3/2}),
\qquad \xi\to+\infty.
\]
Since $g_{km}(\xi)$ consists of products of such Hankel transforms and
$\alpha_{HS}(\xi,\omega)=\mathcal O(\xi^{-1})$ as $\xi\to\infty$, we obtain
\[
\alpha_{HS}(\xi,\omega)\, g_{km}(\xi)\,\xi
= \mathcal O(\xi^{-3}),
\qquad \xi\to\infty.
\]
\par To deal with the pole at $\xi = \xi_R$,  we subtract the singular part of the integrand and add its contribution exactly. We previously used a similar technique to justify the limit $\mathbf{\widehat{S}}(\omega; 0^+)$. In both cases, we use the Laurent expansion around the pole to define $B(\xi, \omega)$. 
\begin{equation}
    B(\xi,\omega) = \alpha_{HS}(\xi,\omega) - \frac{A_R(\omega)}{\xi - \xi_R},
\end{equation}
where $A_R(\omega)$ is the residue of $\alpha_{HS}$ at $\xi_R$.
In section 4.3, we used a small indentation in our path of integration to avoid the cusps and pole, making the integrand analytic in the region enclosed by the contour, so we could apply Cauchy's integral theorem. Here, we do not use any contours in the complex plane. Instead, we separate $S_{km}(\omega + 0^+)$ into a well-defined classical integral of $\widetilde{B}(\xi, \omega)$ (which we approximate) plus the principal value integral of $\frac{g_{km}(\xi_R) \ \xi_R}{\xi - \xi_R}$ and the residue term from the limit as $\eta \to 0^+$ (both of which we compute exactly).
\begin{equation}
    \widehat{S}_{km}(\omega+0^{+}) = \int_{0}^{\xi_{\text{tail}}} B(\xi,\omega) g_{km}(\xi) \xi \, d\xi + A_R(\omega) \ \text{P.V.} \int_{0}^{\xi_{\text{tail}}} \frac{g_{km}(\xi) \ \xi}{\xi - \xi_R} \ d\xi + i\pi A_{R}(\omega) g_{km}(\xi_{R}) \xi_{R}.
\end{equation}
In the compact interval $[0, \xi_{\text{tail}}]$, part of the principal value integral can be evaluated analytically.
\begin{equation}
    \PV \int_{0}^{\xi_{\text{tail}}} \frac{g_{km}(\xi) \xi}{\xi - \xi_R} \, d\xi = \int_{0}^{\xi_{\text{tail}}} \frac{g_{km}(\xi)\xi - g_{km}(\xi_R)\xi_R}{\xi - \xi_R} \, d\xi + \PV \int_{0}^{\xi_{\text{tail}}} \frac{g_{km}(\xi_R)\xi_R}{\xi - \xi_R} \, d\xi
\end{equation}
The first integrand extends continuously across $\xi=\xi_R$ (and is smooth away from the branch points), so we modify $B(\xi, \omega)$ to include this term. Let
\[
\widetilde{B}(\xi, \omega) = B(\xi, \omega)g_{km}(\xi)\xi + A_R(\omega) \frac{g_{km}(\xi) \ \xi - g_{km}(\xi_R) \ \xi_R}{\xi - \xi_R}.
\]
The principal value integral can be evaluated explicitly, yielding
\begin{equation}
    \text{P.V.} \int_{0}^{\xi_{\text{tail}}} \frac{g_{km}(\xi_R) \ \xi_R}{\xi - \xi_R} \ d\xi = g_{km}(\xi_R) \ \xi_R \ \log\left|\frac{\xi_{\text{tail}} - \xi_R}{\xi_R}\right|
\end{equation}
Thus, we have
\begin{equation}
    \widehat{S}_{km}(\omega+0^{+}) = \int_{0}^{\xi_{\text{tail}}} \widetilde{B}(\xi,\omega) \, d\xi + A_R(\omega) g_{km}(\xi_R) \ \xi_R \ \left( \log\left|\frac{\xi_{\text{tail}} - \xi_R}{\xi_R}\right| + i\pi \right).
\end{equation}
The integrand $\widetilde{B}(\xi,\omega)$ is analytic for $\omega>0$
except at square-root branch points at $\xi = \pm k_L, \pm k_T$. The singularities at these points result in unbounded derivatives, and the standard a priori error estimates for the Gaussian quadrature would result in algebraic convergence rates.
To recover spectral convergence, we partition the integration domain $[0, \xi_{\text{tail}}]$ into four intervals: $\Gamma_1 = [0, k_L]$, $\Gamma_2 = [k_L, \xi_{\text{mid}}]$, $\Gamma_3 = [\xi_{\text{mid}}, k_T]$, and $\Gamma_4 = [k_T, \xi_{\text{tail}}]$, where $\xi_{\text{mid}} = (k_L + k_T)/2$. Each branch point $k_L$ and $k_T$ is treated from both sides by splitting the integration domain so that it appears once as a left endpoint and once as a right endpoint. For the purpose of estimating the Gauss-Legendre convergence, it is sufficient to identify, for each subinterval, the singularity of the transformed integrand that lies closest to the real interval $u\in[-1,1]$ in the complex plane. More distant singularities do not affect the admissible Bernstein ellipse and therefore do not influence the exponential convergence rate. This allows the use of quadratic coordinate transformations $[a, b] \mapsto [-1, 1]$, under which the integrand becomes analytic on $u \in [-1, 1]$. We define two mappings, $T_a$ for a left-endpoint singularity and $T_b$ for a right-endpoint singularity:
\begin{align}
    T_a(u; a, b) &= a + (b-a)\left(\frac{1}{2}(u+1)\right)^2, \quad \frac{d T_a}{du} = \frac{1}{2}(b-a)(u+1), \\
    T_b(u; a, b) &= b - (b-a)\left(\frac{1}{2}(u+1)\right)^2, \quad \frac{d T_b}{du} = -\frac{1}{2}(b-a)(u+1).
\end{align}
\begin{remark}
\label{remark:negative_map}
The second coordinate transformation is decreasing in $u$, the interval is reflected.
\[
T_b(-1; a, b) = b, \quad T_b(1; a, b) = a.
\]
The bounds of integration would be reversed,
\[
\int_{a}^{b} f(\xi) \, d\xi = \int_{1}^{-1} f(T_b(u; a, b)) \frac{d T_b}{du} \, du = \int_{-1}^{1} f(T_b(u; a, b)) \left(-\frac{d T_b}{du}\right) \, du.
\]
So for the remainder of the section, we will write the integral with $\left(-d T_b/du\right)$ to fix the orientation of the integral.
\end{remark}

The benefit of substituting $\xi = T_a(u; a, b)$ or $\xi = T_b(u; a, b)$ is to transform the integrand on each interval to be analytic in a complex neighborhood of $[-1, 1]$ and thus analytic in a region of the complex plane enclosed by a Bernstein ellipse, so that we can apply the standard error estimates for Gauss-Legendre quadrature with exponential convergence.

We follow the analysis of Trefethen and define the Bernstein ellipse $E_{\rho}$ with foci at $\pm 1$ as the image of the circle $|z| = \rho > 1$ under the Joukowsky map
\[
E_{\rho} = \left\{ \frac{1}{2} \left( z + \frac{1}{z} \right) : |z| < \rho \right\}.
\]
\paragraph{$\Gamma_1$}
On the interval $\Gamma_1 = [0, k_L]$, the singularity lies at the right endpoint. We use the substitution $\xi = T_b(u; 0, k_L)$ to map the interval $[0, k_L]$ to $u \in [-1, 1]$. We have four singularities in the integrand to consider. Two come from the endpoint factor $\sqrt{\xi^2 - k_L^2}$, and the others from $\sqrt{\xi^2 - k_T^2}$. We consider the singularities separately.
\begin{align}
    -k_L = T_b(u; 0, k_L) = k_L - k_L\left(\frac{1}{2}(u+1)\right)^2 &\implies u = \pm 2\sqrt{2} - 1, \\
    k_L = T_b(u; 0, k_L) = k_L - k_L\left(\frac{1}{2}(u+1)\right)^2 &\implies u = -1.
\end{align}
The singularity we find for the plus case $u = -1$ is removed by the coordinate transformation. The square root factor in the variable $u$
\[
\sqrt{\xi^2-k_L^2}=\left(\frac{1}{2}(u+1)\right)\,\sqrt{k_L^2\left(\frac{1}{2}(u+1)\right)^2-2k_L^2}.
\]
The other points to consider are outside the interval $[0, k_L]$, but here we care about the admissible Bernstein ellipse, which extends past the intervals, so we must consider all singularities in the $u$-plane. The factor $\sqrt{\xi^2 - k_T^2}$ has singularities at $\xi = \pm k_T$.
\begin{align}
    -k_T = T_b(u; 0, k_L) = k_L - k_L\left(\frac{1}{2}(u+1)\right)^2 &\implies u = \pm 2\sqrt{\frac{k_T + k_L}{k_L}} - 1, \\
    k_T = T_b(u; 0, k_L) = k_L - k_L\left(\frac{1}{2}(u+1)\right)^2 &\implies u = \pm 2i\sqrt{\frac{k_T - k_L}{k_L}} - 1.
\end{align}
A similar computation for the other intervals shows that for all $\omega > 0$, there is a finite number of singularities $S_p = T^{-1}(\{\pm k_L, \pm k_T\}) \setminus \{-1\}$, all of which are a positive distance away from the interval $[-1, 1]$ in the $u$-plane (the singularity at $u = -1$ is removed by the coordinate transformation). Therefore, when $\omega > 0$, the integrand $\widetilde{B}(\xi(u), \omega)$ on each interval $\Gamma_p$ is analytic in a complex neighborhood of $[-1, 1]$ in the $u$-plane.

We remind the reader that the boundary of the Bernstein ellipse $E_{\rho}$ is given by the Joukowsky map of the circle $|z| = \rho$.
\begin{remark}
    The Joukowsky map is invariant under $z \mapsto z^{-1}$. Therefore, the inverse admits two values $z_{1}$ and $z_{2}$ such that $z_{2} = 1/z_{1}$.
\end{remark}
In other words, given a point $u_p$ in the $u$-plane corresponding to a singularity, we can find two points in the $z$-plane that map to $u_p$ under the Joukowsky map.
If $u_p$ is a point in the $u$-plane, then its preimages $z_p$ are given by the inverse Joukowsky map:
\begin{equation}
    z_p = u_p \pm \sqrt{u_p^2 - 1}
\end{equation}
The absolute value of these points $|u_p \pm \sqrt{u_p^2 - 1}|$ gives us the radii of the two circles in the $z$-plane, which under the Joukowsky map both map to a single Bernstein ellipse whose boundary passes through the singularity $u_p$. Therefore, we can choose the greater of these two radii to be our radius $\rho_p$. 
Let $\rho(u)$ map a point $u$ in the $u$-plane to the radius of the Bernstein ellipse whose boundary passes through $u$.
\[
\rho(u) = \max\left(\left|u + \sqrt{u^2 - 1}\right|, \left|u - \sqrt{u^2 - 1}\right|\right)
\]
For each interval $\Gamma_p$, we can define the largest Bernstein ellipse $E_{\rho_p}$ that contains no singularities by
\begin{equation}
    \rho_p = \min_{u \in S_p} \rho(u).
\end{equation}

\begin{claim}
    For all $\omega > 0$ and interval $\Gamma_p$, the radius $\rho_p$ defined above satisfies $\rho_p > 1$. 
\end{claim}
\begin{proof}
    Let $u_p$ be the singularity in $S_p$ that minimizes $\rho(u)$. 
    \[
    u_p = \argmin_{u \in S_p} \rho(u).
    \]
    We know that there are two points $z_{p1}, z_{p2}$ in the $z$-plane that map to $u_p$ under the Joukowsky map. Due to the symmetry of the map, we know that $z_{p2} = 1/z_{p1}$. Therefore, if one of these points is inside the unit circle, the other must be outside. We choose the maximum, so $\rho(u_p) > 1$. 
    
    $z_{p1}$ and $z_{p2}$ cannot both be on the unit circle, because if they were, then $u_p$ would be in the interval $[-1, 1]$. However, we know that all singularities $u_p$ lie outside this interval for all $\omega > 0$. Therefore $\rho_p > 1$.
\end{proof}

We define the transformed integrands $F_P(u, \omega)$. Here we write $\left(-d T_b / du \right)$ to fix the orientation of the integral (see Remark \ref{remark:negative_map}).
\begin{align}
    F_{1}(u, \omega) &= -\widetilde{B}\left(T_b(u; 0, k_L), \omega \right) \frac{d T_b}{du}, \quad u \in [-1, 1], \\
    F_{2}(u, \omega) &= \widetilde{B}\left(T_a(u; k_L, \xi_{\text{mid}}), \omega \right) \frac{d T_a}{du}, \quad u \in [-1, 1], \\
    F_3(u, \omega) &= -\widetilde{B}\left(T_b(u; \xi_{\text{mid}}, k_T), \omega \right) \frac{d T_b}{du}, \quad u \in [-1, 1], \\
    F_{4}(u, \omega) &= \widetilde{B}\left(T_a(u; k_T, \xi_{\text{tail}}), \omega \right) \frac{d T_a}{du}, \quad u \in [-1, 1].
\end{align}
Where $F_P(u, \omega)$ is analytic for $\omega > 0$ and $u \in E_{\rho_p}$.

Standard results for Gauss-Legendre quadrature state that if the integrand is analytic inside the Bernstein ellipse $E_{\rho}$ for some $\rho > 1$, then the error in approximating the integral by $N$-point Gauss-Legendre quadrature decays exponentially with $N$ at a rate determined by $\rho$.
The claim above shows that for each interval $\omega > 0$, there exists a Bernstein ellipse $E_{\rho_p}$ for each interval $\Gamma_p$ such that the integrand $F_P(u, \omega)$ is analytic inside $E_{\rho_p}$. 

\begin{remark}
    As $\omega \to 0$, $k_L, k_T \to 0$ which causes the singularities in the $u$-plane for $\Gamma_2$ and $\Gamma_4$ to approach the real axis at $u = -1$. This causes the corresponding preimages in the $z$-plane to approach $|z| = 1$, which causes $\rho \to 1$. Thus, for small frequencies, the convergence rate will slow down. In practice, we compute the S-matrix at $\omega = 0$ separately using the limit of $\alpha_{HS}(\xi, \omega)$ as $\omega \to 0$. It can be shown that
    \[
    \lim_{\omega \to 0} \alpha_{HS}(\xi, \omega) = \frac{C_{L}^{2}}{2 \mu \xi \left(C_{L}^{2} - C_{T}^{2}\right)}
    \]
    which means
    \[
    \lim_{\omega \to 0} \widehat{S}_{km}(\omega + 0^{+}) = \frac{C_{L}^{2}}{2 \mu \left(C_{L}^{2} - C_{T}^{2}\right)} \int_{0}^{\infty} g_{km}(\xi) \, d\xi
    \]
    and $g_{km}(\xi)$ is entire. In this case, we use the same truncation $\xi_{\text{tail}}$ but use the Gauss-Legendre quadrature with points scaled to $[0, \xi_{\text{tail}}]$ to compute the integral. 
\end{remark}

\paragraph{Quadrature Rule}
We generate a Gauss-Legendre quadrature rule for each of the four intervals. The total number of quadrature points is $N = N_1 + N_2 + N_3 + N_4$. We approximate the integral with the sum
\[
\sum_{p=1}^{4} \sum_{j=1}^{N_p} w_{j,p} F_{p}(u_{j}, \omega) 
\]
where the grid points $\{u_{j}\}_{j=1}^{N_p}$ are zeros of the $N_p$-th Legendre polynomial on $[-1, 1]$ and the weights $\{w_{j,p}\}_{j=1}^{N_p}$ are the standard Gauss-Legendre weights.
\paragraph{A priori Error}
Define the Gauss-Legendre quadrature error in the interval $\Gamma_P$ as
\begin{equation}
    R_{P} = \int_{-1}^{1} F_{P}(u, \omega) \, du - \sum_{j=1}^{N_p} w_{j,P} F_{P}(u_{j}, \omega)
\end{equation}
Since endpoint transformations remove square root cusps, the transformed integrand $F_P(u,\omega)$ is analytic on $u \in [-1, 1]$ and extends analytically to the open region enclosed by $E_{\rho_p}$. Standard results on the Gauss quadrature for analytic functions (see Trefethen~\cite[Thm.19.3]{Trefethen2019}) then imply geometric convergence. Specifically, we have the error estimate:
\begin{equation}
\label{eq:error}
    |R_{p}| \leq \frac{64}{15} \frac{M_{p} \rho_p^{-2N_p}}{\rho_p^2 - 1}
\end{equation}
Where $M_{p} = \max_{z \in E_{\rho_p}} |F_{P}(z, \omega)|$.

\begin{figure}
    \centering
    \includegraphics[width=0.75\linewidth]{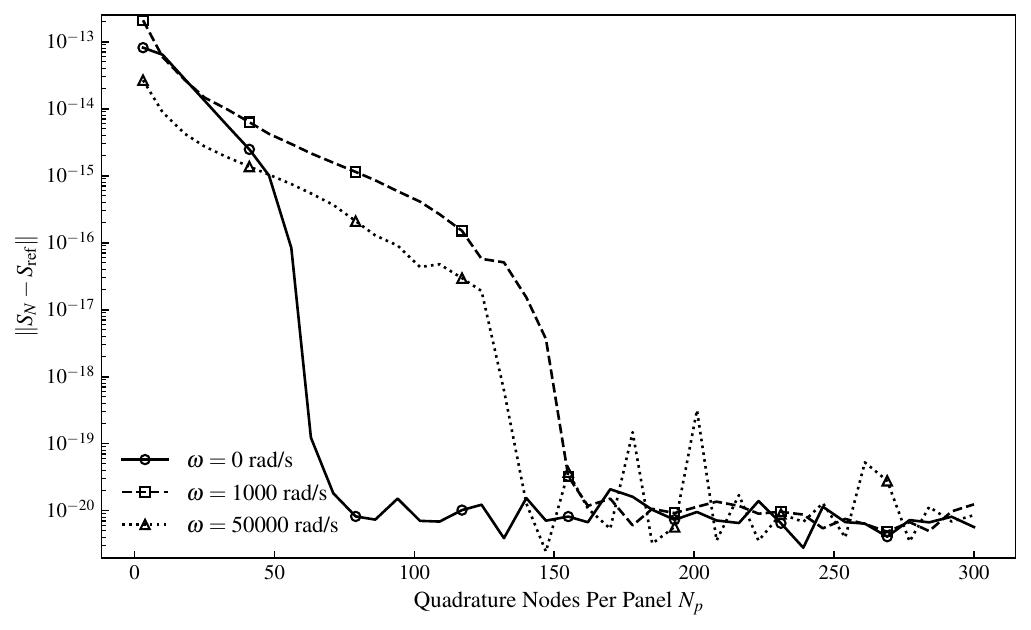}
    \caption{Spectral convergence of the discretized $\mathbf{\widehat{S}}$-matrix to a reference solution $\mathbf{\widehat{S}}_{\mathrm{ref}}$. $\xi_{\text{max}} = 2200, R = 0.0762$m}
    \label{fig:S_conv}
\end{figure}

In the absence of an exact solution, we compute a reference solution $\mathbf{\widehat{S}_{\mathrm{ref}}}$ with a large number of nodes. For these and all subsequent numerical experiments, we adopt the physical constants defined in Chen et al.~\cite[Table 2]{Chen1988}. We estimate the error in a solution $\mathbf{\widehat{S}}_N$ at a fixed $\omega$ by $\|\mathbf{\widehat{S}}_N(\omega) - \mathbf{\widehat{S}}_{\mathrm{ref}}(\omega)\|$\footnote{The matrix norm used here is $\|A\| = \sqrt{\sum_i \sum_j |a_{ij}|^2}$}. Geometric convergence in $N$ is observed, consistent with~\eqref{eq:error}.
For smaller $R$ this regime is reached for fewer nodes, since the integrand is less oscillatory. This is illustrated by the log-linear plot~\ref{fig:S_conv}, which has three curves corresponding to $\omega=0$, $\omega = 1000$, and $\omega = 50000$. Before reaching the machine precision limit, the plot shows two distinct regimes of convergence, highlighting the impact of oscillations in the integrand. 

Lesser values of $N$ (visually estimated as $N < 120$ for the dynamic case and $N < 50$ for the static case) show the pre-asymptotic region, characterized by geometric convergence but at a rate significantly slower than what the theoretical $\rho$ predicts. This slower convergence is likely due to aliasing caused by the highly oscillatory integrand. 
Once there are enough points to sample the oscillations sufficiently, we observe a much steeper convergence that is much closer in slope to the predicted $\rho$.

In further support of this aliasing hypothesis, figure ~\ref{fig:R_conv} has four lines corresponding to varying $R$ values. As $R$ increases, the integrand becomes more oscillatory due to terms like $J_0(R\xi)$; in turn, it takes a larger $N$ to resolve the highly oscillatory integrand.
\begin{figure}[!h]
    \centering
    \includegraphics[width=0.75\linewidth]{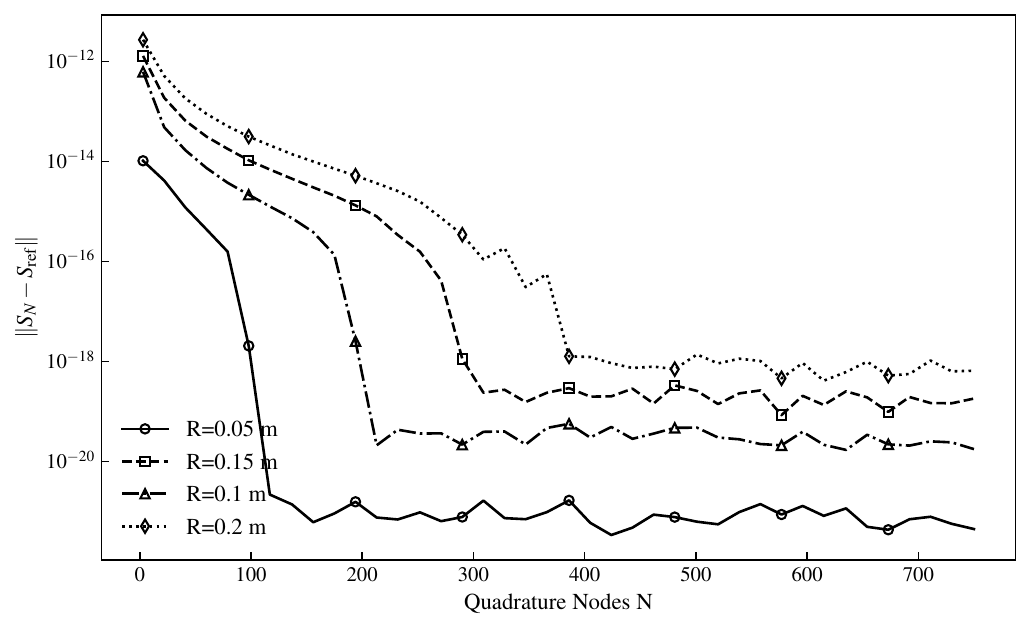}
    \caption{Error over number of quadrature nodes for various values of $R$, for fixed $\omega = 1000$ and $\xi_{\text{max}} = 2800$}
    \label{fig:R_conv}
\end{figure}
Furthermore, it appears that the error saturation floor rises as $R$ increases. This is also likely a result of the worsening oscillations for larger $R$. As $R$ increases, not only does the frequency of the oscillations increase, but the magnitude of the oscillations also increases. This introduces catastrophic cancellation errors inherent to floating-point arithmetic that cannot be surpassed by a more refined quadrature.

\subsection{Inversion and Numerical Stability}
After computing all entries of the $\widehat{\mathbf{S}}$ matrix via the quadrature scheme described above, we can solve the linear system in~\eqref{eq:sys_S} for the coefficients $C_{nm}$. We reconstruct the functions $\psi_n(r)$ from these coefficients.
\[
\psi_n(r) = \sum_{m=1}^{N} C_{nm} \phi_m(r)
\]
Finally, we solve for the expansion coefficients $a_n$ of the displacement by solving the system:
\begin{equation}
    \sum_{n=1}^{N} a_n K_{nm} = \frac{\widehat{p}(\omega)}{2\pi} \phi_m(0)
\end{equation}
\begin{equation}
    K_{mn} = \left(D \lambda_n^4 - \rho h \omega^2\right) N_{mn} + \int_{0}^{R} \psi_n(r) \phi_m(r) r \, dr
\end{equation}

For small $R$ ($R < 1$), the condition number is small for both matrices, and the systems were solved using LU decomposition with partial pivoting, implemented in the LAPACK routine \texttt{GESV}~\cite{LAPACK}
For larger $R$, the condition number of $\mathbf{\widehat{S}}$ increases, and we found it more stable to accept a least squares solution via the LAPACK routine \texttt{gelsd}\footnote{Both routines were accessed via SciPy's \texttt{scipy.linalg} module~\cite{Scipy}.}.

\section{Finite-Radius Results}
We now present representative numerical results illustrating the effect of finite plate radius on the plate-half-space interaction. Figures \ref{fig:small_R_strain} and \ref{fig:Large_R_strain} display the radial strain over time at a fixed point $r=12.7$ mm for two distinct radii.

Figure \ref{fig:small_R_strain} presents the radial strain in a plate of radius $R=76.2$ mm, chosen to match the experimental setup detailed in \cite{Chen1988}. 
The plate loading~\ref{eq:plate_time}
\[
    p(t) = \frac{1}{2} F_0\left(1-\cos\left(\frac{2 \pi t}{T_0}\right)\right), \quad t \in [0, T_0]
\]
where $F_0$ is the amplitude, and $T_0$ is the contact duration both chosen to be  the values determined in~\cite{Chen1988}.
The computed signal exhibits good qualitative agreement with the experimental data reported in \cite{Chen1988}. The signal captures the initial peak from the impact, followed by oscillations. These later-time oscillations are characteristic of the finite plate; they correspond to waves reflecting off the free edge $r=R$ and returning to the observation point.

\begin{figure}[!h]
    \centering
    \includegraphics[width=0.75\linewidth]{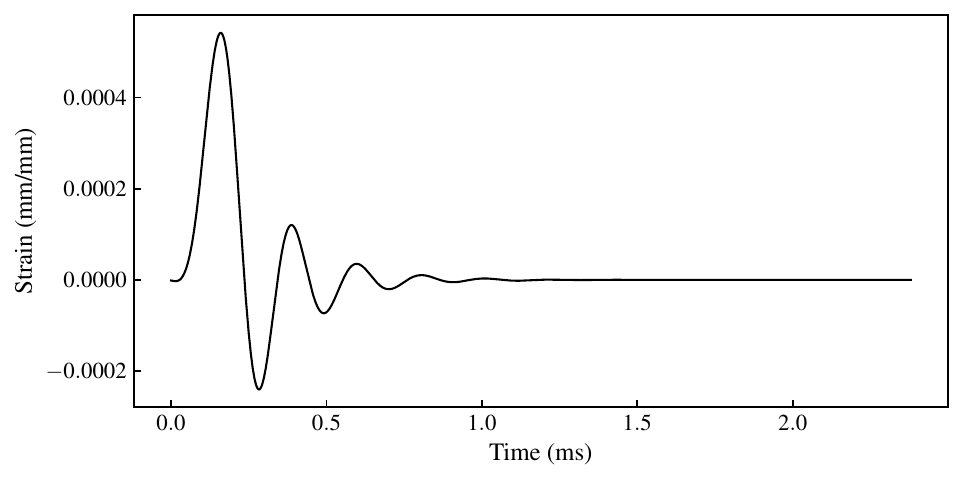}
    \caption{Radial strain in a plate of radius  $R=76.2$mm at the point $r=12.7$mm over time.}
    \label{fig:small_R_strain}
\end{figure}
 \subsection{Comparison With the Infinite-Plate Model}
Figure \ref{fig:Large_R_strain} presents the response for a large plate $R=1000$mm, also observed at fixed $r=12.7$mm. This strain signal has a single peak at the initial impact and closely matches the infinite-radius plate analysis conducted in \cite{Chen1988}. Notably, we do not observe the oscillations seen in the small radius case. More pronounced geometric spreading results in a much weaker wave reaching the boundary. Furthermore, the waves continuously radiate energy into the half-space as they travel toward the boundary.
\begin{figure}[!h]
    \centering
    \includegraphics[width=0.75\linewidth]{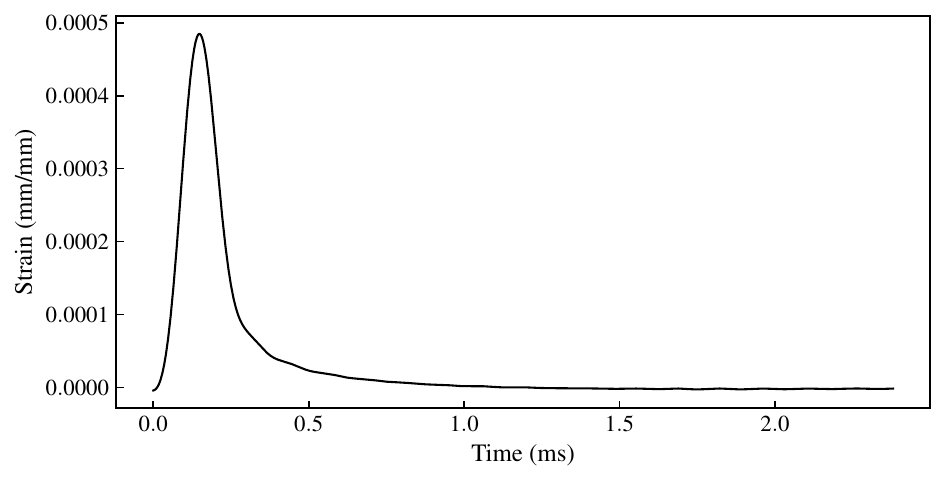}
    \caption{Radial strain in a plate of radius  $R=1000$mm at the point $r=12.7$mm over time.}
    \label{fig:Large_R_strain}
\end{figure}

Spectral expansions ideally yield exponential convergence for analytic functions. However, in our numerical experiments, we observe algebraic decay in the coefficients $a_n(\omega)$ (the graph makes roughly a straight line on the log-log plot in figure~\ref{fig:a_n_conv}). This behavior is due to the singular nature of the contact problem at the edge of the plate. The plate-half-space interaction is a mixed boundary value problem on the soil; the traction field $\widehat{q}(r, \omega)$ is only supported on the contact region. The result, as established in classical contact mechanics~\cite[\S3.4]{Johnson1985}, is typically a square-root singularity of the form $(R^2 - r^2)^{-1/2}$ at the boundary of the contact region ($r=R$).

\begin{figure}[!h]
    \centering
    \includegraphics[width=0.75\linewidth]{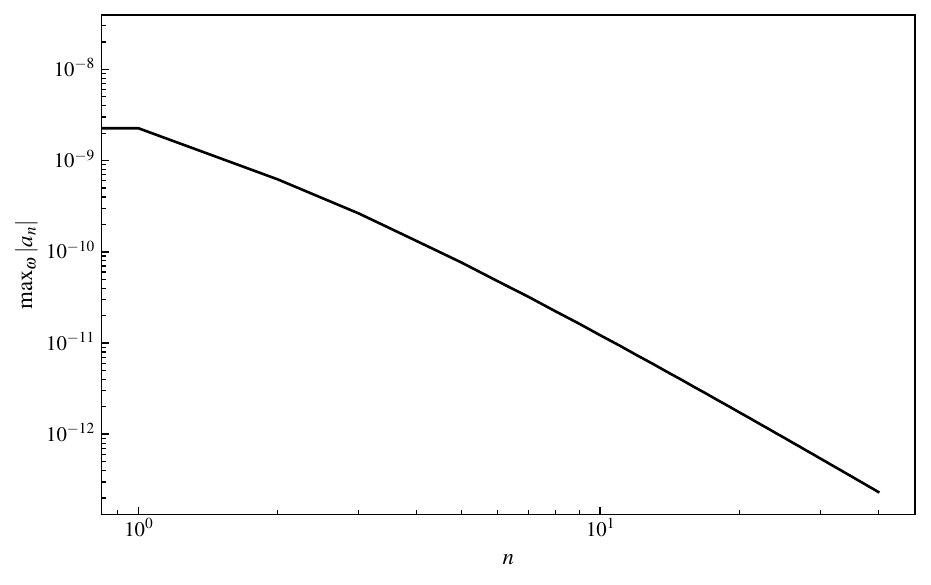}
    \caption{Decay of series coefficients $a_n$ for $R=76.2$ mm}
    \label{fig:a_n_conv}
\end{figure}

The convergence rate of an expansion is limited by regularity. Attempting to approximate a singular function using a basis of smooth functions inevitably results in algebraic decay of the coefficients ($O(n^{-k})$) rather than exponential decay, as well as persistent oscillations (Gibbs phenomenon) in the reconstructed traction field near the discontinuity $r=R$.

To connect with the classical theory, we examine the large-radius limit. As $R\to\infty$, $\chi_{[0,R]}$ converges strongly to the identity on $L^2(0,\infty)$ (with radial weight), and $\mathcal M(\omega)$ approaches the
infinite-domain Lamb operator $T(\omega)$. In this regime, the off-diagonal entries of the modal matrix become small, and the formulation asymptotically recovers the diagonal structure of the infinite-plate model. We quantify the
finite-radius correction by comparing representative entries and operator norm surrogates of $\widehat{\mathbf S}(\omega)$ against their infinite-domain counterparts, and show that the discrepancy is most pronounced at low modes and low to moderate frequencies (figure~\ref{fig:s_diff}), where truncation effects are not negligible. We compute $\widehat{\mathbf S}(\omega)$ twice, once with $R = 800$ mm and the other computed with $R=76.2$, denoted $S_1$ and $S_2$ respectively.
\begin{figure}[!h]
    \centering
    \includegraphics[width=0.75\linewidth]{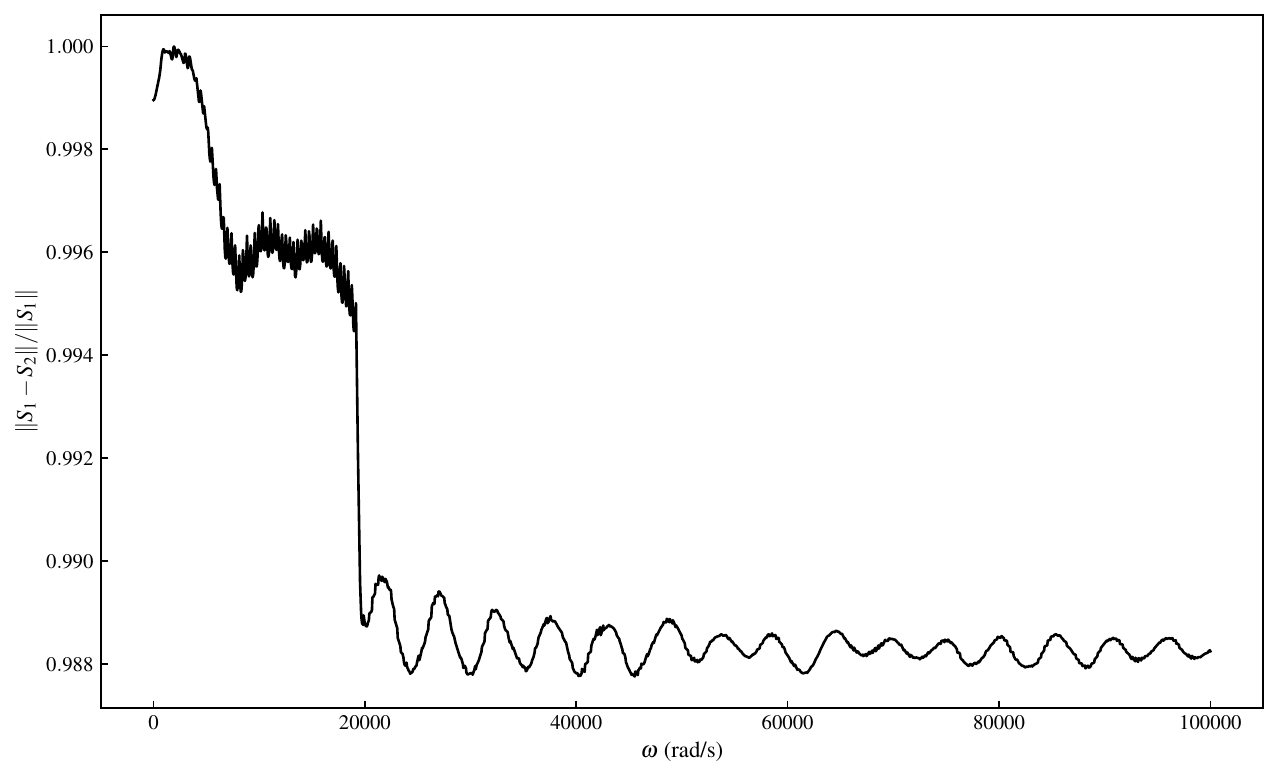}
    \caption{Frobenius norm of relative difference between $S_1$ and $S_2$}
    \label{fig:s_diff}
\end{figure}

\subsection{Displacement}
We next translate the operator formulation into physically interpretable quantities. For each driving frequency $\omega$, the coupled plate-half-space system yields the frequency-domain plate deflection $\widehat w(r,\omega)$ on $0\le r\le R$. We report both pointwise and integral response measures, such as the center displacement $w(0,t)$ and the kinetic energy $K(\omega)$ of the thin plate. We define the kinetic energy:
\[
K(\omega) = \pi \rho h \omega^2 \int_0^R |\widehat{w}(r, \omega)|^2 r dr.
\]

\begin{figure}[!h]
    \centering
    \includegraphics[width=0.75\linewidth]{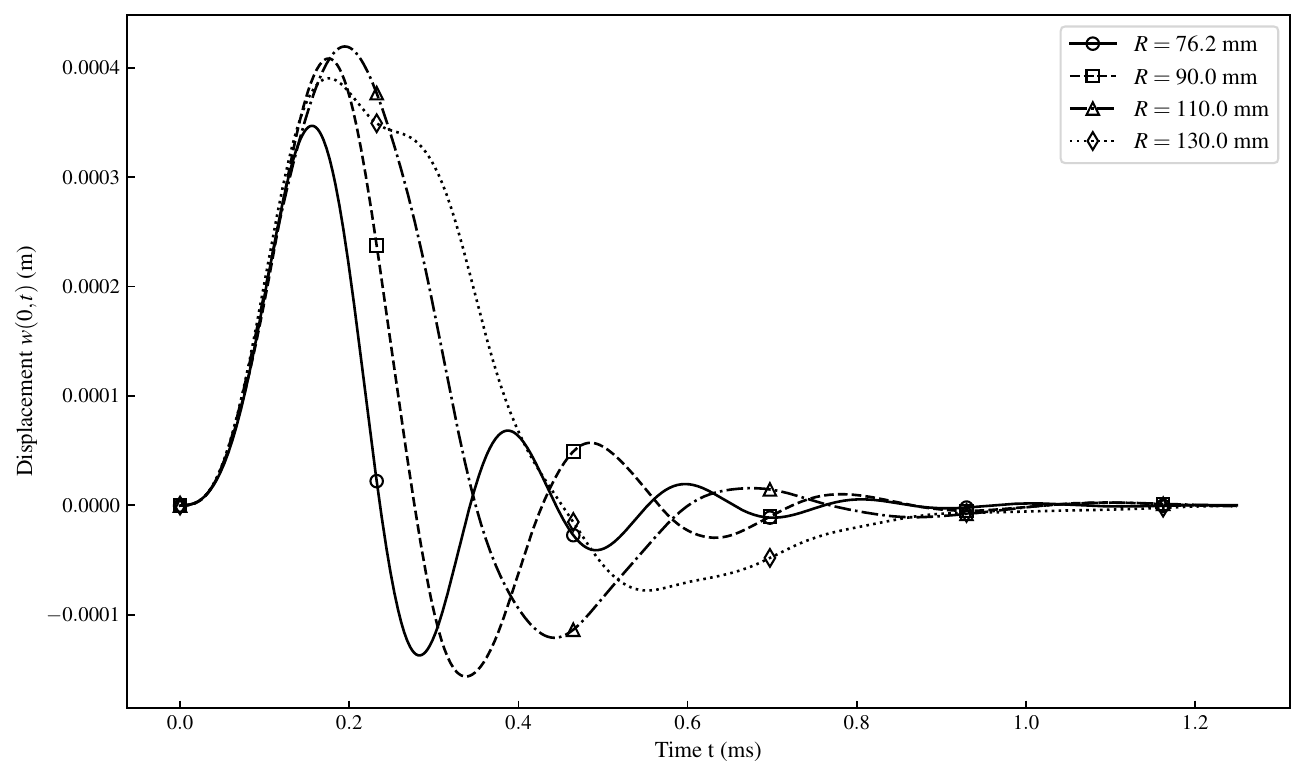}
    \caption{Center displacement over time for multiple values of $R$}
    \label{fig:center_displacement}
\end{figure}

\begin{figure}[!h]
    \centering
    \includegraphics[width=0.75\linewidth]{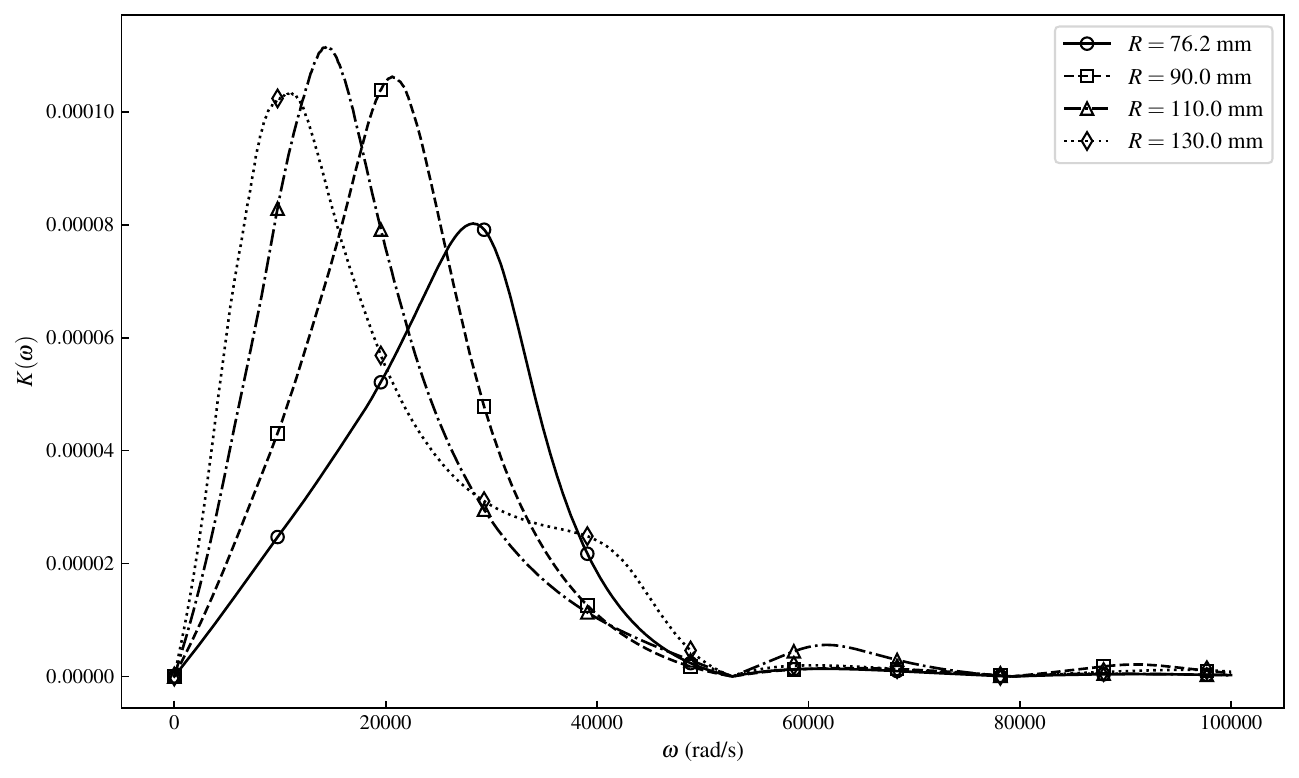}
    \caption{Kinetic energy of the plate in the frequency domain}
    \label{fig:energy}
\end{figure}

Figure~\ref{fig:center_displacement} reports the center displacement $w(0, t)$ over time for various values of $R$. As $R$ increases, the oscillations after the initial impact are delayed because the reflection off the boundary must travel further. The amplitude of the oscillations also decreases as $R$ increases due to radiating more energy into the soil. Figure~\ref{fig:energy} reports the kinetic energy of the plate in the frequency domain. As $R$ increases, the energy is more concentrated at the lower frequencies. We also reconstruct the resulting displacement in the soil using Lamb's solution
\[
\widehat{w}^{\,0}(\xi, \omega) = \widehat{q}^{\,0}(\xi, \omega)\frac{\alpha}{\mu \Omega} \left[\left(2 \xi^2 - k_T^2 \right)e^{-\alpha z} - 2\xi^2e^{-\beta z} \right]
\]
with our computed traction. After inverting the transforms, we plot wave propagation in the soil in figure~\ref{fig:soil_wave}.
\begin{figure}[!h]
    \centering
    \includegraphics[width=0.75\linewidth]{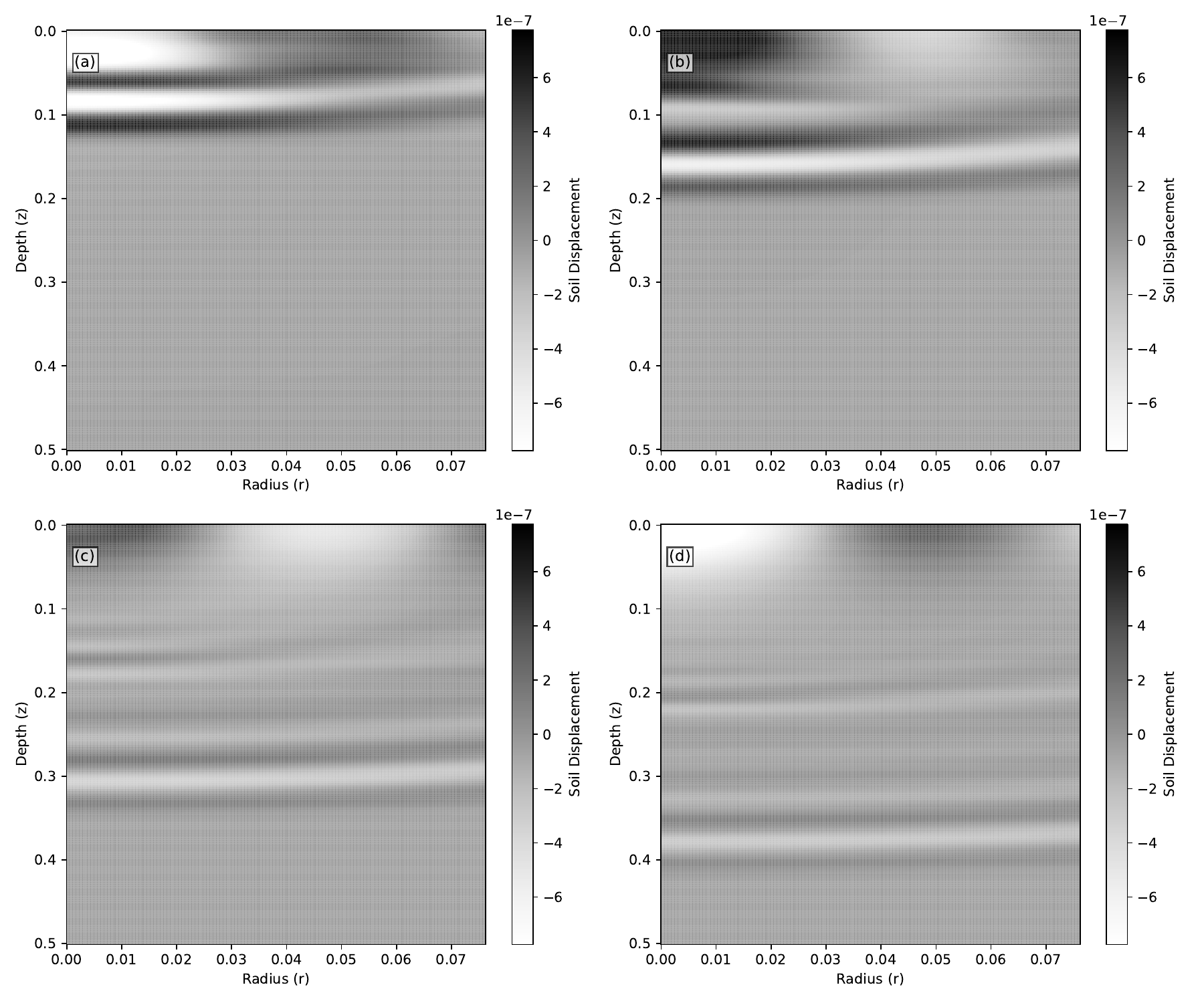}
    \caption{Displacement of soil at times (a) $t = 0.60$ ms, (b) $0.89$ ms, (c) $1.49$ ms, and (d) $1.79$ ms}
    \label{fig:soil_wave}
\end{figure}

\section{Frequency-Domain Amplitude Comparison}
In figure~\ref{fig:strain_omega}, we examine the strain amplitude in the frequency domain and compare it with Chen et al.’s experimental results. For large radii $R$, the computed signal is very close to the strain under the infinite plate assumption. For a smaller $R$, the computed signal indicates less high-frequency vibration than the experimental data (\cite[Fig. 8]{Chen1988}). The observed differences are not surprising and are most likely due to experimental factors such as measurement noise. These effects tend to influence the high-frequency portion of the spectrum more strongly. In contrast, the corresponding time-domain strain signals show much better agreement in the physically relevant part of the response, suggesting that the model accurately captures the system's dominant dynamics.

\begin{figure}[!h]
    \centering
    \includegraphics[width=0.75\linewidth]{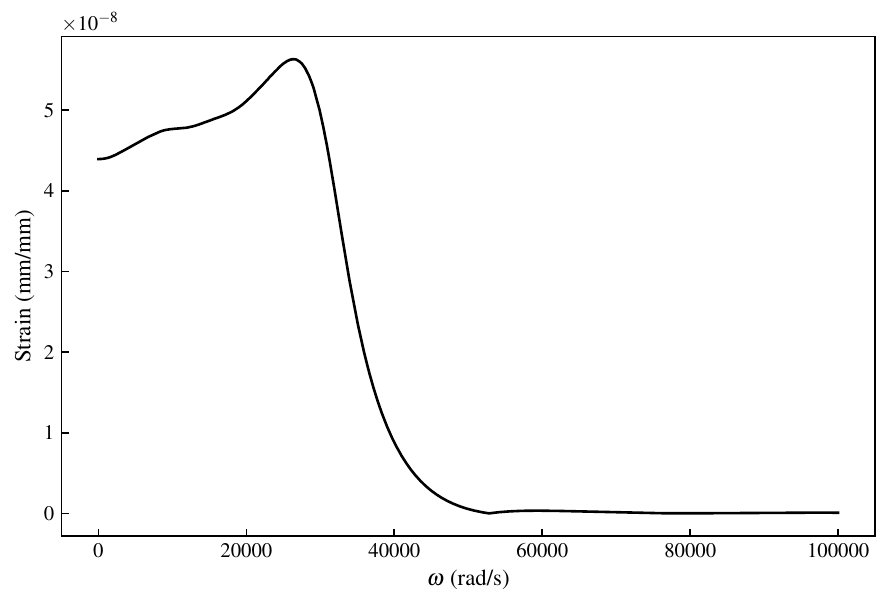}
    \caption{Radial strain in a plate of radius $R=76.2$ mm at the point $r=12.7$ in frequency}
    \label{fig:strain_omega}
\end{figure}

In this work, Fourier transforms are evaluated by using a direct numerical integration of the defining continuous integral rather than a fast Fourier transform (FFT). Although FFTs provide an efficient way to compute discrete Fourier transforms of sampled data, both approaches approximate the same underlying continuous transform when consistent normalization and windowing conventions are used. A detailed, fully quantitative comparison with experimental frequency-domain data, accounting for preprocessing choices and noise characteristics, is beyond the scope of this study and will be addressed in future work.

\section{Discussion and Extensions}
The operator-based formulation developed here makes the essential differences between finite-radius configurations and the infinite-plate model transparent. Restricting the Lamb operator to a bounded disk eliminates the Hankel diagonalization available on the full half-line and leads to genuinely nonlocal coupling between modes. At the same time, the analytic structure of the Lamb kernel introduces radiation effects through its poles and branch cuts. Both features are inherent in physically realistic geometries for which the infinite-domain approximation is not always appropriate.

A natural direction for future work is a detailed spectral analysis of the truncated operator $\mathcal M(\omega)$, including its eigenvalue structure and its dependence on the radius and frequency of the plate. The present framework can also admit possible natural extensions. One direction is the treatment of non-axisymmetric loading, which can be incorporated by including higher-order angular Fourier–Bessel modes. Finally, the frequency-domain formulation can be inverted to the time domain to investigate transient loading and impact problems, where the combined effects of modal coupling and radiation damping are
expected to play a significant role. These directions can be pursued in future work.

Another potential avenue for future work is the restoration of spectral convergence. The algebraic decay observed is a consequence of expanding a singular traction field in a smooth Bessel basis. This limitation could be addressed by using a mapped grid designed for endpoint singularities.

\section{Conclusions}
In this work, we develop an exact operator formulation for the dynamic
interaction of a finite circular elastic plate with an elastic half-space
by introducing the truncated Lamb operator
$\mathcal M(\omega)=\chi_{[0,R]}T(\omega)\chi_{[0,R]},$ acting on radial functions supported on $[0,R]$. Representing the operator $\mathcal M(\omega)$ on a finite-disk Bessel basis leads to a matrix representation in which the truncation to $[0,R]$ prevents diagonalization, seen in the infinite-radius setting. As a result, no individual mode evolves independently, and the resulting modal matrix contains nonzero off–diagonal entries.

Although this matrix is dense, its entries decay rapidly with increasing mode number, and the representation converges spectrally. Each matrix element is expressed as an improper integral, interpreted via the limiting absorption principle, which separates the reactive response of the half-space from explicit radiation contributions associated with Rayleigh surface waves and bulk-wave emission. In the limit as the radius of the plate $R\to\infty$, the formulation converges to the infinite plate model.

This work aims to develop a framework that provides a foundation for finite–radius plate–soil interaction problems well suited to systematic extensions to more general geometries, boundary conditions, and constitutive models.\\

\noindent
{\bf Acknowledgments.} The authors thank Professor Hung-Liang Chen (Department of Civil and Environmental Engineering, West Virginia University) for bringing this problem to our attention and for helpful discussions and clarifications. His willingness to answer detailed technical questions throughout this work is warmly acknowledged.\\

\noindent
{\bf Data Availability.} The data that support the findings of
this study are available from the corresponding author
upon reasonable request.\\

\noindent
{\bf Conflict of Interest Statement.}
The authors declare that they have no conflict of interest.
\printbibliography
\end{document}